\documentclass[reqno]{amsart}
\usepackage[scale=0.75, centering, headheight=14pt]{geometry}
\usepackage[latin1]{inputenc}
\usepackage[T1]{fontenc}
\usepackage{lmodern}
\usepackage[english]{babel}
\usepackage{stmaryrd}
\usepackage{tikz}
\usepackage{bbm}
\usepackage{amsmath,amssymb,amsfonts,amsthm}
\usepackage{mathtools,accents}
\usepackage{mathrsfs}
\usepackage{xfrac}
\usepackage{array} 
\usepackage{aliascnt}
\usepackage{booktabs} 
\usepackage{array} 

\usepackage{verbatim} 
\usepackage{subfig} 
\usepackage{mathrsfs}
\usepackage{mathrsfs, dsfont}
\usepackage{amssymb}
\usepackage{amsthm}
\usepackage{amsmath,amsfonts,amssymb,esint}
\usepackage{graphics,color}
\usepackage{enumerate}
\usepackage{mathtools,centernot}

\usepackage{microtype}
\usepackage{paralist} 
\usepackage{cases}
\usepackage[initials]{amsrefs}
\allowdisplaybreaks

\usepackage{braket}
\usepackage{bm}

\usepackage[citecolor=blue,colorlinks]{hyperref}
\addto\extrasenglish{}

\usepackage{enumerate}
\usepackage{xcolor}

\usepackage{aliascnt}

\makeatletter
\def\newaliasedtheorem#1[#2]#3{
  \newaliascnt{#1@alt}{#2}
  \newtheorem{#1}[#1@alt]{#3}
  \expandafter\newcommand\csname #1@altname\endcsname{#3}
}
\makeatother

\theoremstyle{plain}
\newtheorem{theorem}{Theorem}[section]
\newaliasedtheorem{lemma}[theorem]{Lemma}
\newaliasedtheorem{prop}[theorem]{Proposition}
\newaliasedtheorem{claim}[theorem]{Claim}
\newaliasedtheorem{corollary}[theorem]{Corollary}
\newtheorem*{NTEO}{Theorem}

\theoremstyle{remark}

\theoremstyle{definition}
\newaliasedtheorem{definition}[theorem]{Definition}
\newaliasedtheorem{example}[theorem]{Example}

\newtheorem{OQ}[]{Open Question}

\theoremstyle{remark}
\newaliasedtheorem{remark}[theorem]{Remark}
\newaliasedtheorem{ex}[theorem]{Example}

\numberwithin{equation}{section}

\def\eps{\varepsilon}

\def\R{\mathbb R}
\def\N{{\mathbb N}}

\DeclareMathOperator{\dv}{div}

\DeclareMathOperator{\spn}{span}
\DeclareMathOperator{\curl}{curl}

\DeclareMathOperator{\Ker}{Ker}

\DeclareMathOperator{\rank}{rank}
\DeclareMathOperator{\loc}{loc}

\DeclareMathOperator{\op}{op_{\mathbb{H}}}
\DeclareMathOperator{\opr}{opr_{\mathbb{H}}}

\newcommand{\PO}{\mathbb{P}_\mathbb{H}}
\newcommand{\POL}{\mathbb{P}}

\DeclareMathOperator{\co}{co}

\DeclareMathOperator{\im}{Im}

\newcommand{\weak}{\overset{*}{\rightharpoonup}}

\DeclareMathOperator{\diam}{diam}
\DeclareMathOperator{\dist}{d}

\newcommand{\A}{\mathcal{A}}

\newcommand{\OP}{\op (k,m,n,N)}
\newcommand{\OPR}{\opr (k,m,n,N)}

\newcommand{\PH}{\mathbb{P}_H}
\title{The four-state problem and convex integration for linear differential operators}

\author[M. Sorella and  R. Tione]{Massimo Sorella \and Riccardo Tione}
\address{Massimo Sorella
\hfill\break EPFL B, Station 8, CH-1015 Lausanne, CH}
\email{massimo.sorella@epfl.ch}
\address{Riccardo Tione  
\hfill\break  EPFL B, Station 8, CH-1015 Lausanne, CH}
\email{riccardo.tione@epfl.ch}

\begin{document}

\maketitle

\begin{abstract}
We show that the four-state problem for general linear differential operators is flexible. The only flexibility result available in this context is the one for the five-state problem for the $\curl$ operator due to B. Kirchheim and D. Preiss, \cite[Section 4.3]{KIRK}, and its generalization \cite{FS}. To build our counterexample, we extend the convex integration method introduced by S. M\"uller and V. \v Sver\'ak in \cite{SMVS} to linear operators that admit a potential, and we exploit the notion of \emph{large} $T_N$ configuration introduced by C. F\"orster and L. Sz{\'{e}}kelyhidi in \cite{FS}.
\end{abstract}

\par
\medskip\noindent
{\sc MSC (2020): 35B99 - 35E20 - 35G35.
\par
}

\section{Introduction} 

In the last years, much attention has been given to the study of properties of $\A$-free maps $u$, i.e. maps $u$ that verify $$\mathcal{A}(u) = 0,$$ for a linear differential operator $\mathcal{A}$. The properties one can expect from $\A$-free maps are strongly related to the form of $\A$. If it is elliptic, for instance $$\A(u) = \Delta u,$$ one can show smoothness of solutions. If $\A$ is not elliptic one cannot expect in general any improvement in the regularity of $u$. For example, if $\A = \curl$, the best one can infer on an $\A$-free map $u$ is that it can be locally expressed as the gradient of a map $v$. In this case, more subtle questions arise: for instance, understanding the structure of the singular part of $\curl$-free measures leads to the celebrated rank-one theorem of Alberti, \cite{ALB}, see also \cite{MV,GUIANN}. In this paper, we shall focus on non-elliptic operators.
\\
\\
Classically, the most studied non-elliptic operator is $\A = \curl$, and we refer the reader to \cite{DMU} for an account of the theory. A rich literature is also available in the case $\A = \dv$, see for instance \cite{SER,GN,PP,DRST,LR}. Typical questions in this context concern fine qualitative properties of measures $\mu$ satisfying $\mathcal{A}(\mu) = 0$, see \cite{GUIANN, RDHR}, higher order estimates and regularity, see \cite{SER,LR,HIGHER, VS, GC}, semicontinuity of functionals defined on $\A$-free maps, see \cite{FM,DRST,REL,WIE,RAI}, and structural results on Young measures generated by sequences of $\A$-free maps, see \cite{FM,GR,KR}.
\\
\\
This paper is devoted to the study of the $s$-state problem for general linear operators $\A$, that we state now. Fix an open set $\Omega \subset \R^{m}$ and a linear differential operator $\A:C^{\infty}(\Omega , \R^{n}) \to  C^{\infty}(\Omega , \R^{N})$ of order $k$ and consider its associated wave cone, $\Lambda_\A \subset \R^n$, see \eqref{LAMBDA} for the definition. It is well-known that if $a,b \in \R^n$ satisfy
\[
a-b \in \Lambda_\A,
\]
then one can find a non-constant oscillatory solution $u$ to
\[
\begin{cases}
u(x) \in \{a,b\}, &\text{ a.e. on }\Omega\\
\mathcal{A}(u) = 0, &\text{ in the sense of distributions}. 
\end{cases}
\]
This can be achieved by the so-called simple laminate construction, see the beginning of Section \ref{section:simplelaminates}. The question becomes more challenging when we add the constraint
\[
a-b \notin \Lambda_\A.
\]
Therefore, the $s$-state problem precisely asks whether there exists a non-constant solution to
\begin{equation}\label{problem}
\begin{cases}
u(x) \in \{a_1,\dots, a_s\}, &\text{ a.e. on }\Omega\\
\mathcal{A}(u) = 0, &\text{ in the sense of distributions}\\
a_i-a_j \notin \Lambda_\A, & \text{ if } i\neq j, 1\le i,j \le s.
\end{cases}
\end{equation}
System \eqref{problem} has already received much attention, and we give now an account of the literature.
\\
\\
Problem \eqref{problem} was first studied for $\A = \curl$. In that context, J.M. Ball and R.D. James in \cite{BJ} have shown that if $s = 2$, the problem is \emph{rigid}, i.e. the only solution to $\eqref{problem}$ is the constant one. The same rigidity holds if $s = 3$, and is sometimes attributed to K. Zhang as in \cite{CMKB,SMA} and sometimes to V. \v Sver\'ak, as in \cite[Section 2.4]{DMU}. Rigidity still holds for $s = 4$, as proved in \cite{CMKB} by Kirchheim and M. Chleb\'ik. Finally, for $s = 5$, the problem becomes \emph{flexible}, i.e. one can find a non-constant map $u$ that takes precisely $5$ states and solves \eqref{problem}. This construction is due to Kirchheim and Preiss and appears in \cite[Section 4.3]{KIRK}. Similar results are known also for $\A = \dv$. For this operator, rigidity for the $s$-state problem \eqref{problem} was proved by A. Garroni and V. Nesi in \cite{GN} and by M. Palombaro and M. Ponsiglione in \cite{PP}, in the case $s = 2$ and $s = 3$ respectively. To the best of our knowledge, nothing is known for $s \ge 4$. Some results concerning rigidity for linear operators $\A$ of order one also appeared in \cite{MAB}. Finally, for general operators $\A$, rigidity for $s = 2$ was proved by G. De Philippis, L. Palmieri and F. Rindler in \cite{PPR}.
\\
\\
Our main theorem fits in this list of results, since it asserts that the four-state problem is flexible:
\begin{NTEO}
There exists an operator $\A$ such that problem \eqref{problem} with $s = 4$ admits a non-constant solution.
\end{NTEO}
Our main result should be compared with \cite[Theorem 1.2(A)]{PPR}. In particular, it states that a result in the generality of \cite[Theorem 1.2(A)]{PPR} is not possible if $s \ge 4$. The case $s = 3$ remains open. For such $s$, the only known result is the rigidity for operators of order one, that can be inferred from the rigidity result for $\A = \dv$ of \cite{PP}, as we will prove in Proposition \ref{rig}. Therefore, we can list here the following open questions on the problem:
\begin{OQ}
Is problem \eqref{problem} rigid for operators of order $2$ or higher if $s = 3$?
\end{OQ}
\begin{OQ}
Is problem \eqref{problem} rigid for operators of order $1$ if $s = 4$?
\end{OQ}

Together with the study of \emph{exact} solutions to \eqref{problem}, one may consider rigidity and flexibility of \emph{approximate} solutions to \eqref{problem}, i.e. classify limit points of sequence $u_n$ equibounded in $L^\infty$ and satisfying
\begin{equation}\label{prob_app}
\begin{cases}
\dist(u_n(x),\{a_1,\dots, a_s\}) \to 0, &\text{strongly in }L^1 \text{ as }n\to \infty,\\
\mathcal{A}(u_n) = 0, \forall n, &\text{ in the sense of distributions},
\end{cases}
\end{equation}
coupled once again with the requirement $a_i - a_j \notin \Lambda_\A$ if $i \neq j$. In this case, if $\A = \curl$, the problem \eqref{prob_app} is rigid, i.e. $(u_n)_n$ converges strongly in $L^1$ to a constant, if $s = 2,3$, see \cite{BJ,SMA}, but it is flexible if $s = 4$, due to the existence of Tartar's $T_4$, see for instance \cite[Lemma 2.6]{DMU}. In \cite[Theorem 1.2(B)]{PPR}, it is shown that the same rigidity result holds for general operators $\A$ if $s = 2$. This is sharp, since in \cite[Lemma 4.1]{GN} the authors show flexibility of {\em approximate} solutions for the operator $\A$ by producing a $T_3$ configuration for the operator $\A = \dv$.
\\
\\
Let us outline the strategy we adopt to show our main theorem. We recall that in \cite{FS}, F\"orster and Sz{\'{e}}kelyhidi gave the definition of \emph{large} $T_5$ \emph{configuration}, that generalizes Kirchheim and Preiss construction of \cite{KIRK}. Every large $T_5$ configuration is a five-point set $K \subset \R^n$ that fulfills some geometric constraints and has the property that the $5$-state problem for $K$ with $\A = \curl$ is non-rigid. Firstly, we extend this notion in a natural way for general linear operators, see the Definition \ref{T4} of large $\Lambda_\A$-$T_4$ configurations. Instead of fixing a particular operator $\A$ and then trying to find a large $\Lambda_\A$-$T_4$ configuration, we consider a set $K = \{a_1,a_2,a_3,a_4\}$ that satisfies suitable geometric constraints and then we find an operator $\A$ such that $K$ is a large $\Lambda_\A$-$T_4$ configuration for that particular operator. In order for this plan to work, we will need to prove that, as in \cite{FS}, large $\A$-${T_4}$ configurations yield non-constant solutions of \eqref{problem}. A large part of this proof comes from the convex integration framework introduced essentially by M\"uller and \v Sver\'ak in \cite{SMVS} in the case of the $\curl$ operator. The reason why M\"uller and \v Sver\'ak develop these methods is to find counterexample to regularity of critical points of quasiconvex energies. Since then, these techniques have been successfully applied in various contexts and for various operators, compare \cite{LSP, LINF,COR}, and they still are one of the main tools for trying to build counterexamples, see \cite{LOP,JTR,DLDPKT}. M\"uller and \v Sver\'ak's theory is systematically developed only for the $\curl$ operator, and we extend it to homogeneous linear operators \emph{of constant rank}. One of the main ingredients we will use is the notion of \emph{potential} introduced by B. Rai{\c{t}}{\u{a}} in \cite{RAI}.
\\
\\
Let us end this introduction by giving an outline of the paper. In Sections \ref{not} and \ref{prel}, we introduce the notation and collect some preliminaries on linear operators. Section \ref{CI} is devoted to develop all the tools of M\"uller and \v Sver\'ak's approach to convex integration in the case of general linear operators that admit a potential. In Section \ref{fours} we define large $\Lambda_\A$-$T_4$-configurations and we show our main theorem by finding a counterexample to the four state-problem. Finally, in the Appendix we will show the rigidity of the three-state problem for operators of order $1$.
\\
\\
\textbf{ Acknowledgements}. 
The authors wish to thank Federico Stra for suggesting to use the second method explained in Proposition \ref{computer}. The authors have been supported by the SNF Grant 182565. 

\section{Notation}\label{not} We define $\mathcal{M}(d,m)$ to be the space of multi-indexes $I = (\alpha_1,\dots,\alpha_m) \in \mathbb{N}^m$ with $|\alpha_1| + |\alpha_2| + \dots +|\alpha_m| = d$. $\PH (d,m)$ defines the vector space of homogeneous polynomials of degree $d$ in $\R^m$. With the notation above, an element $p \in \PH(d,m)$ can be written as
\[
p(x) = \sum_{I \in \mathcal{M}(d,m)}a_Ix^I,
\]
where for $x = (x_1,\dots, x_m) \in \R^m$ the notation $x^I$ means
\[
x^I = x_1^{\alpha_1}x_2^{\alpha_2}\dots x_m^{\alpha_m}.
\]
We also introduce $\POL(d,m)$, the vector space of polynomials of degree $d$ in $\R^m$.
\\
\\
$\Omega \subset \R^m$ will always be used to denote an open bounded set. A function $f: \Omega \to \R$ is said to be \emph{piecewise a polynomial of degree $d$} if there exists a countable family of pairwise disjoint open sets $\{\Omega_n\}_n$ such that
\[
\left|\Omega \setminus \bigcup_n \Omega_n\right| = 0
\]
and, on $\Omega_n$, every component of $f$ is a polynomial of degree $d$. The definition of \emph{piecewise smooth} is analogous. Throughout the paper, $|E|$ denotes the Lebesgue measure of a measurable $E \subset \R^m$. 
\\
\\
We will say that $E \subset \R^m$ is \emph{essentially open in} $\Omega$ if $\left|\partial E\cap \Omega \right| = 0$. Here, $\partial E$ is the topological boundary of $E$. $\overline{E}$ denotes the closure of $E$. We denote by $B_{\eps}(E)$ the $\eps$-neighbourhood of the set $E$ and by $\co(E)$ the convex hull of $E$. For two elements $a,b \in \R^n$, we will use the notation $[a,b]$ for $\co(\{a,b\})$.
\\
\\
The set of probability measures compactly supported in $U \subset \R^n$ is denoted by $\mathcal{P}(U)$. We let $\bar \nu \doteq \int_{\R^n} x d\nu(x)$ be the barycentre of $\nu \in \mathcal{P}(U)$.

\section{Preliminaries on general linear operators}\label{prel}

Let $\A$ be a differential operator acting on vector-valued functions  $v \in C^{\infty}(\Omega ; \R^{n})$, where $\Omega \subset \R^{m}$ is an open set, namely
 \begin{equation} \label{d_operator}
 \A v \doteq \sum_{\ell = 1}^k\sum_{\alpha \in \mathcal{M}(\ell,d)} A_\alpha \partial^{\alpha} v + Av + C(x),
 \end{equation}
here, $A_\alpha, A \in \R^{N \times n}$ are constant matrices and $C \in L_{\loc}^{2} (\Omega; \R^{N})$. Note that the equation $\A v =0 $ is actually a {\em system of $N$ equations}.
 We will use the notation $\text{op}(k,m,n,N)$ to denote these operators, but we will actually always consider \emph{homogeneous} differential operators, i.e. $C(x) \equiv 0$ and $A = 0$, $A_{\alpha} = 0$, if $|\alpha| < k$ in \eqref{d_operator}. The set of homogeneous operators will be denote by $\OP$.
\\
\\
Let $\A \in \op (k,m,n,N)$. For each $\xi \in \R^m$, we consider the linear maps $\mathbb{A}(\xi): \R^n \to \R^N$ defined as
\begin{equation}\label{ell}
\mathbb{A}(\xi)(\eta) \doteq \sum_{\alpha \in \mathcal{M}(k,m)}\xi^\alpha A_\alpha\eta,\quad \forall \eta \in \R^n.
\end{equation}
Define the {\em wave cone} associated to $\A$ as:
\begin{equation}\label{LAMBDA}
\Lambda_{\A} \doteq  \bigcup_{\xi \in \R^m\setminus\{0\}}\Ker(\mathbb{A}(\xi)) =  \{\eta \in \R^n: \exists \xi \in \R^m\setminus\{0\} \text{ s.t. } \mathbb{A} (\xi)(\eta) = 0\}.
\end{equation}
In what follows, we will only consider operators $\mathcal{A} \in \op (k,m,n,N)$ with \emph{constant rank}, namely
\[
\xi \mapsto \rank(\mathbb{A}(\xi)) \text{ is constant}.
\]
This class of operators will be denoted with the symbol $\OPR$. We will exploit \cite[Theorem 1]{RAI}, that asserts that the homogeneous operator $\mathcal{A}$ is of constant rank if and only if it admits a potential (of constant rank), meaning that there exists $\mathcal{B} \in \op(k',m,n',n)$ such that
\begin{equation}\label{pot}
\Ker(\mathbb{A}(\xi)) = \im(\mathbb{B}(\xi)), \quad\forall \xi \in \R^m\setminus\{0\}.
\end{equation}

For technical reasons, in Section \ref{CI} we will need to restrict ourselves to \emph{balanced} operators, that we now introduce.

\subsection{Balanced Operators}

In addition to the constant rank condition, we require an additional property on the linear differential operator $\A$.

\begin{definition}\label{balanced}
We say that the wave cone $\Lambda_\A$ is \emph{balanced} if 
\begin{equation}
\spn(\Lambda_\A) = \R^n,
\end{equation}
and we say that an operator $\A \in \OP$ is \emph{balanced} if the associated wave cone $\Lambda_\A$ is \emph{balanced}.
\end{definition}

The heuristic idea for which we need to consider balanced operators stems from the fact that on $\spn(\Lambda_{\A})^{\perp}$, the operator is, in some sense \emph{elliptic}, compare \cite[Equation (4)]{GRA}. This can be seen clearly in the extreme case $\spn(\Lambda_\A) = \{0\}$, in which one has
\[
\mathcal{A}(u) = 0 \Rightarrow u \in C^\infty(\Omega, \R^n).
\]
Since we are interested in constructing irregular solutions via convex integration, the images of these will surely avoid directions contained in $\spn(\Lambda_{\A})^{\perp}$. However, the requirement that $\mathcal{A}$ is balanced is mainly made for simplicity of exposition and is in fact not restrictive. Indeed we have the following simple result:
\begin{prop} \label{p:balanced}
Let $\mathcal{A} \in \OP$. Let $\pi \doteq \spn(\Lambda_\A)$ and let $d \ge 1$ be its dimension. Fix an orthonormal basis $e_1,\dots, e_d$ for $\pi$. Then, if we define $\mathcal{A}' \in \op(k,m,d,N)$ as
\begin{equation}\label{notres}
\mathcal{A}'(u) \doteq \mathcal{A}\left(\sum_{i = 1}^du_ie_i\right), \text{ if } u = \left(\begin{array}{c}u_1\\ \vdots \\ u_d\end{array}\right),
\end{equation}
the following hold:
\begin{itemize} \label{condition:spanbal}
\item $\mathcal{A}'$ is balanced;
\item $\mathcal{A}$ has constant rank if and only if $\mathcal{A}'$ has.
\end{itemize}
\end{prop}

The proposition tells us that we may study wild solutions of the balanced operator $\mathcal{A}'$ instead of studying those of $\mathcal{A}$, and by \eqref{notres} these will be also solutions to $\mathcal{A}(u) = 0$. We omit the proof of Proposition \ref{p:balanced} since the verifications are simple.
\\
\\
The fact that $\A$ is balanced yields the following:

\begin{prop}\label{surj}
Let $\mathcal{A} \in \OPR$ be balanced, and let $\mathcal{B} \in \op(k',m,n',n)$ be a potential for $\mathcal{A}$. Then, the map $T :(\PH(k',m))^{n'} \to \R^n$  defined as $T(q) \doteq \mathcal{B}(q)$ is surjective.
\end{prop}

\begin{proof}[Proof of Proposition \ref{surj}]
The proof is by contradiction: suppose $T$ is not surjective. Fix $\xi \in \R^m$ and $a \in \R^{n'}$. We choose the polynomial
\[
p(x) \doteq \sum_{I \in \mathcal{M}(k',m)}\frac{\xi^{I}}{|I|}x^I
\]
and define $q(x) \doteq p(x) a \in (\PH(k',m))^{n'}$. A direct computation shows that
\[
T(q) = \mathbb{B}(\xi)(a).
\]
This yields
\begin{equation}\label{Bpsi}
\im(\mathbb{B}(\xi)) \subset \im(T),\qquad \forall \xi \in \R^m.
\end{equation}
In particular, since $T$ is linear and not surjective, we find a non-zero vector $v \in \R^n$ such that
\[
v \perp \im(T),
\]
and, using \eqref{Bpsi},
\[
v \perp \im(\mathbb{B}(\xi)), \qquad \forall \xi \in \R^m.
\]
Since $\mathcal{B}$ is the potential of $\mathcal{A}$, by definition \eqref{pot} holds, and we find a contradiction with the definition of $\mathcal{A}$ being balanced.
\end{proof}
\section{Convex integration for general differential operators of constant rank}\label{CI}

Throughout the section, we will consider a fixed balanced operator $\mathcal{A} \in \OPR$, with a given potential $\mathcal{B} \in \op(k',m,n',n)$. 
\\
\\
Aim of this part of the work is to develop the convex integration scheme essentially due to M\"uller and \v Sver\'ak in the case of the $\curl$ operator, see for instance \cite[Sections 2,3]{SMVS}. The final goal is to being able to show the existence of a non-constant solution $u \in L^\infty(\Omega,\R^n)$ to the following system:
\begin{equation}\label{AINC}
\begin{cases}
u(x) \in K, \; \text{a.e. in }\Omega,\\
\mathcal{A}(u) = 0,
\end{cases}
\end{equation}
where $\mathcal{A} \in \OPR$, $\Omega$ is a given open, bounded, convex set, and $K \subset \R^n$ is a compact set without $\Lambda_\A$ connections, i.e. for any $a,b \in K$ we have that $b-a \notin \Lambda_\A$. In the case of the four-state problem that we will treat in Section \ref{fours}, $K$ is the four-point set of the admissible states. In particular, our aim is to show that the existence of a $\mathcal{A}$-in-approximation $\{U_n\}_n$ of $K$, see Definition \ref{INAPP}, yields the existence of a (in fact, many) non-constant solutions to \eqref{AINC}.
\\
\\
Due to the technical nature of some  proofs of this section, it is probably better to briefly explain our strategy. First, in Subsection \ref{section:simplelaminates}, we introduce the building blocks of this convex integration scheme, the simple $\mathcal{A}$-laminates. Roughly speaking, these are highly oscillatory solutions of \eqref{AINC} that can be constructed starting from two vectors $a,b \in \R^n$ with $b-a \in \Lambda_\A$. Their properties are listed in Proposition \ref{lam+}. Subsequently, we define $\mathcal{A}$-laminates of finite order, and describe their main properties, see Definition \ref{d:laminates_finite} and Proposition \ref{ind}. Then, we move on to $\mathcal{A}$-laminates, see Subsection \ref{section:laminates}, and we quote a result of \cite{KIRK} that asserts the weak-$*$ density of $\mathcal{A}$-laminates of finite order in the space of $\mathcal{A}$-laminates, compare Theorem \ref{ann}. We will use this result in Section \ref{inappexsol} to show the preliminary Proposition  \ref{usefulprop} and finally Theorem \ref{rcexact}, that asserts the existence of exact solutions to \eqref{AINC} once we are given a $\mathcal{A}$-in-approximation.

\subsection{Simple laminates}  \label{section:simplelaminates} 
The building block is given by the simple $\A$-laminate construction. Let $a,b \in \R^n$ be such that $$b-a = c \in \Ker(\mathbb{A}(\xi_0)) \subset \Lambda_{\mathcal{A}}.$$ It is simple to check that for any profile $h \in L^\infty(\R)$, the map
\[
v(x) \doteq h((x,\xi_0))c
\]
solves $\mathcal{A}(v) = 0$. Here and in the following, $(x,y)$ denotes the standard scalar product of $\R^m$. This observation can be refined as follows. Let $\lambda \in (0,1)$ be arbitrary, $e \doteq \lambda a + (1-\lambda) b$, and choose  
\begin{equation}\label{h}
h(t) \doteq \begin{cases}\lambda, &\text{ if } t \in [0,1-\lambda)\\ -(1-\lambda), &\text{ if } t \in [1-\lambda,1],\end{cases}
\end{equation}
and its 1-periodic extension outside $[0,1]$. If we let
\begin{equation}\label{v}
v_{\eps,\xi_0,a,b,\lambda}(x) \doteq e + h\left(\frac{(x,\xi_0)}{\eps}\right)c,
\end{equation}
one can check that, given any bounded open set $\Omega \subset \R^m$, $v_{\eps,\xi_0,a,b,\lambda}$ enjoys the following properties
\begin{enumerate}
\item\label{1} $\mathcal{A}(v_{\eps,\xi_0,a,b,\lambda}) = 0$, $\forall \eps > 0$;
\item\label{2} $|\{x: v_{\eps,\xi_0,a,b,\lambda}(x) = a\}| \to \lambda|\Omega|$ and $|\{x: v_{\eps,\xi_0,a,b,\lambda}(x) = b\}| \to (1-\lambda)|\Omega|$ as $\eps \to 0^+$;
\item\label{3} $v_{\eps,\xi_0,a,b,\lambda} \overset{*}{\rightharpoonup} e$ in $L^\infty$ as $\eps \to 0^+$. 
\end{enumerate}

In other words, every element of the $\Lambda_{\mathcal{A}}$-cone $c$ gives rise to a family of highly oscillatory solutions to the PDE defined by $\mathcal{A}$. The oscillating behaviour is due to the choice of a periodic profile $h$, and yields to the fact that these solutions do not converge strongly, as can be easily seen from \eqref{2}-\eqref{3}.
\\
\\
Using the theoretical potential $\mathcal{B}$, we can find a potential for $v_{\eps,\xi_0,a,b,\lambda}$. Indeed, since $\mathbb{A}(\xi_0)(c) = 0$, by \eqref{pot}, there exists $c' \in \R^{n'}$ such that
\begin{equation}\label{ima}
\mathbb{B}(\xi_0)(c') = c.
\end{equation}
Furthermore, we consider the unique function $H: \R \to \R$ such that $H^{(k')}(t) = h(t)$, $\forall t \in \R$ and $H^{(\ell)}(0) = 0$, $\forall 0 \le \ell \le k' -1$, where $H^{(\ell)}$ denotes the $\ell$-th derivative of $H$, and $H^{(0)} \doteq H$. Finally we choose any $q_e \in \POL(k',m)^{n'}$ such that
\[
\mathcal{B}(q_e) = e, \text{ everywhere on $\R^m$}.
\]
By Proposition \ref{surj}, there exists at least one vector of polynomials with this property. If we define
\[
V_{\eps,\xi_0,a,b,\lambda}(x) \doteq  q_e(x) + \eps^{k'}H\left(\frac{(x,\xi_0)}{\eps}\right)c',
\]
then we see by construction that
\[
\mathcal{B}(V_{\eps,\xi_0,a,b,\lambda})(x) = e + h\left(\frac{(x,\xi_0)}{\eps}\right)\mathbb{B}(\xi_0)(c') = e + h\left(\frac{(x,\xi_0)}{\eps}\right)c = v_{\eps,\xi_0,a,b,\lambda}(x),
\]
almost everywhere and in the sense of distributions. Notice that by construction $V_{\eps,\xi_0,a,b,\lambda}$ is, for every $\eps > 0$, a vector of piecewise polynomials of degree $k'$. This discussion allows us to prove the following:

\begin{lemma}\label{lam}
Let $\Omega\subset \R^m$ be an open and bounded set. Let $a,b \in \R^n$, $b-a  = c \in \Lambda_{\A}$ and $e = \lambda a + (1-\lambda) b$, for some $\lambda \in (0,1)$. Fix any element $q_e \in \POL(k',m)^{n'}$ with the property that $\mathcal{B}(q_e) = e$ everywhere in $\R^m$. Then, for all $\alpha > 0$, there exists $V_\alpha \in W^{k',\infty} \cap C^{k' -1}(\overline{\Omega},\R^{n'})$, and two disjoint open sets $\Omega_\alpha^1$, $\Omega_\alpha^2$ with $|\Omega| = |\Omega_\alpha^1\cup \Omega_\alpha^2|$ such that
\begin{enumerate}
\item\label{f0} the $W^{k',\infty} \cap C^{k' -1}$ norm of $V_\alpha$ only depends on $\diam(\Omega),|a|,|b|$ and $|D^{k'}q_e|$;
\item\label{f1} $V_\alpha = q_e$, together with all its derivatives of order $\ell < k'$, on $\partial\Omega$;
\item\label{s1} Every component of $V_\alpha$ is piecewise a polynomial of degree $k'$,
\item\label{t1} Let $v_{\alpha}(x) \doteq \mathcal{B}(V_\alpha)(x)$. The sets $A_\alpha = \{x \in \Omega_{\alpha}^1: v_\alpha(x) = a\}$, $B_\alpha = \{x \in \Omega_{\alpha}^1: v_\alpha(x) = b\}$, $ \Omega_\alpha^1 \doteq A_\alpha \cup B_\alpha$ and $\Omega_\alpha^2 \doteq (\Omega_1^\alpha)^c$ are essentially open in $\Omega$, , and
\[
|A_\alpha| \ge (1-\alpha)\lambda|\Omega| \text{ and } |B_\alpha| \ge (1 - \alpha)(1-\lambda)|\Omega|.
\]
\item\label{ff1} $|\Omega_\alpha^2| \le \alpha|\Omega|$;
\item\label{fff1} $\|V_\alpha - q_e\|_{C^{k' - 1}}\le \alpha$;
\item\label{ss1} $v_{\alpha}(x) \in B_{\alpha}([a,b])$ a.e. in $\Omega$.
\end{enumerate}
\end{lemma}

\begin{proof}
Fix $\alpha > 0$. Choose an open set $\Omega'$ compactly contained in $\Omega$ with $|\Omega \setminus \Omega'| \le \frac{\alpha}{2}|\Omega|$, $\Omega \setminus \Omega'$ essentially open in $\Omega$, and let $\varphi$ be a fixed smooth cut-off function with values in $[0,1]$ such that $\varphi(x) = 1$, $\forall x \in \Omega'$. With the notation introduced before the statement of the lemma, we define
\[
W_{\eps}(x) \doteq q_e(x) + \eps^{k'}\varphi(x)H\left(\frac{(x,\xi_0)}{\eps}\right)c'.
\]
We wish to take $\Omega_\alpha^1 \doteq \Omega'$, $\Omega_\alpha^2$ as the interior of $\Omega \setminus \Omega_\alpha^1$ and $V_\alpha \doteq W_{\eps}$ for $\eps > 0$ sufficiently small, and up to a correction on the small set $\Omega_\alpha^2$ in order to make every component piecewise polynomial. With these choices \eqref{f1} and \eqref{t1} are immediate, once $\eps$ is chosen sufficiently small, and \eqref{ff1} is a consequence of \eqref{t1}. As $\eps \to 0$, the boundedness in $L^\infty$ of $H$ yields the strong convergence in $L^\infty$ of $W_\eps$ to $q_e$. To see that the convergence is in the $C^{k' - 1}$ topology, it is sufficient to show the equiboundedness in $W^{k',\infty}(\Omega, \R^{n'})$. To see the latter, it is sufficient to take a derivative of order $k'$ of
\[
W'_\eps(x) \doteq \eps^{k'}\varphi(x)H\left(\frac{(x,\xi_0)}{\eps}\right).
\]
Let then $I \in \mathcal{M}(k',m)$. $\partial_IW'_\eps(x)$ can be estimated by a sum of terms of the form
\begin{equation}\label{prodrule}
\eps^\ell\partial_{I'}\varphi(x) H^{(k'-\ell)}\left(\frac{(x,\xi_0)}{\eps}\right),
\end{equation}
where $I' \in \mathcal{M}(\ell,m)$. It is then easy to see that if $\eps = \eps(\alpha)$ is sufficiently small, we may estimate the latter by $\|h\|_{L^\infty}$, and hence conclude that \eqref{f0} holds. A similar computation shows \eqref{fff1}. Finally, with computations analogous to the ones of \eqref{prodrule}, one can estimate:
\begin{equation}\label{close}
\left|\mathcal{B}(W_\eps)(x) - e - \varphi(x)h\left(\frac{(x,\xi_0)}{\eps}\right)c\right| \le C\eps,
\end{equation}
for some constant $C > 0$ at a.e. $x \in \Omega$. Since
\[
e + \varphi(x)h\left(\frac{(x,\xi_0)}{\eps}\right)c \in [a,b],
\]
from \eqref{close} we further deduce \eqref{ss1}. The map $W_\eps$ satisfies all the properties listed in the statement of the lemma, except for \eqref{s1}. It is simple to see, from the definition of $W_\eps$, that on $\Omega_1^\alpha$ every component of $W_\eps$ is piecewise a polynomial of degree $k'$ and that it is globally a piecewise smooth map. Therefore, we may subdivide $\Omega_\alpha^2$ into pairwise disjoint, compactly supported and open cubes $Q_j$ on each of which $W_\eps$ is a smooth map up to the boundary. By Lemma \ref{pp} below, we see that, on every $Q_j$, $W_\eps$ can be substituted with a map $W_{\eps,j}$ whose components are piecewise polynomials of order $k'$ with $W_{\eps,j} = W_\eps$ on $\partial Q_j$ and arbitrarily small $\|W_\eps - W_{\eps,j}\|_{C^k(\overline{\Omega})}$. It is simple to check that if this norm is taken sufficiently small, then \eqref{f0}-\eqref{f1}-\eqref{s1}-\eqref{t1}-\eqref{ff1}-\eqref{fff1} still hold for the map defined as $W_{\eps,j}$ on $Q_j$ and $W_\eps$ everywhere else. This defines the map $V_\alpha$.
\end{proof}

We now show Lemma \ref{pp}, that was used in the previous proof. This states that any map $u \in C^k(\Omega)$ can be finely approximated by functions $v\in C^k(\Omega)$ that are piecewise polynomials of order $k$. This was done in \cite[Proposition 3.3]{KIRK} in the case $k = 2$.

\begin{lemma}\label{pp}
Let $\Omega$ be open and $u \in C^{k'}(\overline{\Omega})$. Then, for all $\eps > 0$, there exists a function $v_\eps \in C^{k'}(\overline{\Omega})$ such that
\begin{enumerate}
\item \label{pp1} $\|u - v_\eps\|_{C^{k'}(\overline\Omega)} \le \eps$;
\item \label{pp2} $v_\eps$ is  piecewise a  polynomial of order $k'$;
\item \label{pp3}$v_\eps = u$ together with all of its derivatives of order $0\le \ell \le k'$ on $\partial \Omega$, $\forall \eps > 0$.
\end{enumerate}
\end{lemma}

\begin{proof}
Fix $\eps > 0$. We obtain $v_\eps$ as the limit of a sequence $v_n$ defined inductively. Set $\eps_n = \frac{\eps}{2^n}$. We claim that, given a function $v_n$ with the following properties:
\begin{enumerate}[(a)]
\item\label{indpp11} $v_n$ is $C^{k'}$ up to the boundary of $\Omega$;
\item\label{indpp12} $\Omega_n^1 \subset \{x: v_n \text{ is piecewise a polynomial of order }k' \text{ in a neighborhood of $x$} \}$ and $\Omega_n^2 = (\Omega_n^1)^c$ are essentially open sets in $\Omega$ with $|\Omega_n^2| \le \prod_{j = 1}^n\eps_j|\Omega|$;
\item\label{indpp13} $v_n = u$ together with all of its derivatives of order $0\le \ell \le k'$ on $\partial \Omega$;
\end{enumerate}
then it is possible to find $v_{n + 1}$ such that 
\begin{enumerate}[(A)]
\item\label{indpp21} $v_{n+1}$ is $C^{k'}$ up to the boundary of $\Omega$;
\item\label{indpp22} $\|v_{n + 1} - v_n\|_{C^{k'}} \le \eps_{n + 1}$;
\item\label{indpp23} $\Omega_{n + 1}^1 \subset \{x: v_{n + 1} \text{ is piecewise a polynomial of order }k' \text{ in a neighborhood of $x$}\}$ and $\Omega_{n + 1}^2 = (\Omega_{n + 1}^1)^c$ are essentially open sets in $\Omega$ such that $|\Omega_{n + 1}^2| \le \prod_{j = 1}^{n + 1}\eps_j|\Omega|$ and $\Omega_{n + 1}^2 \subset \Omega_{n}^2$;
\item\label{indpp24} $v_{n + 1} = u$ together with all of its derivatives of order $0\le \ell \le k'$ on $\partial \Omega$.
\end{enumerate}
If this inductive step holds, then we start with $v_0 \doteq v$, working with the convention that
\[
\sum_{j = 1}^0\eps_j = 0 \text{ and } \prod_{j = 1}^0\eps_j = 1.
\]
Since $\{v_n\}_n$ is a Cauchy sequence with respect to the $C^{k'}$ topology by \eqref{indpp22}, we can define $v_\eps \doteq \lim_n v_n$. It is then easy to see that this function $v_\eps$ has the required properties.
\\
\\
To show the inductive step, we consider $\Omega_n^2$ and we first subdivide it in countably many, compactly contained, pairwise disjoint open cubes such that $|\Omega_n^2\setminus \bigcup_rQ_r| = 0$. On $Q_r$, all the derivatives of $v_n$ are uniformly continuous and hence we can find $\delta > 0$ such that if $x,y \in Q_r$ and $|x-y| \le \delta$, then
\begin{equation}\label{contv}
\sum_{j = 0}^{k'}|D^jv_n(x) - D^jv_n(y)| \le \gamma\eps_{n +1},
\end{equation}
where $\gamma > 0$ is a dimensional constant that will be fixed later. Now further subdivide $Q_r$ as a finite union of cubes $Q_{r,s}$ with $\diam(Q_{r,s}) \le \delta$. Fix a compactly contained open set $S_{r,s} \subset Q_{r,s}$ with 
\begin{equation}\label{crs}
|Q_{r,s}\setminus S_{r,s}| \le \eps_{n + 1}|Q_{r,s}|.
\end{equation}
Finally, fix a smooth cut-off function $\psi \in C^{\infty}_c (Q_{r,s})$ such that $\psi \equiv 1$ on $S_{r,s}$, with 
\begin{equation}\label{estpsi}
\|D^{\ell}\psi\|_{L^\infty} \le \frac{c}{\diam(Q_{r,s})^\ell}, \quad \forall \ell \ge 0,
\end{equation}
where $c> 0$ is a dimensional constant. We modify $v_n$ on $Q_{r,s}$ by replacing it with
\[
(1-\psi(x))v_n(x) + \psi(x)P_{r,s}(x),
\]
where $P_{r,s}$ is the $k'$-th order Taylor polynomial centred in the center of $Q_{r,s}$. This operation defines $v_{n + 1}$. Now \eqref{indpp21}-\eqref{indpp24} are immediate to check. \eqref{indpp23} follows by construction and \eqref{crs}, noticing that $\Omega_{n + 1}^1 = \Omega_n^1 \cup \bigcup_{r,s} S_{r,s} $. We only need to show \eqref{indpp22}. We check \eqref{indpp22} separately on every $Q_{r,s}$. For all $x \in Q_{r,s}$, we have, for every multi-index $I \in \mathcal{M}(\ell,m)$, $0 \le \ell \le k'$:
\begin{align*}
|\partial_I(v_{n + 1} - v_n)(x)| = |\partial_I((1-\psi(x))v_n(x) + \psi(x)P_{r,s}(x) - v_n(x))| = |\partial_{I}(\psi(x)(P_{r,s}(x) - v_n(x)))|.
\end{align*}
With a triangle inequality, it is easy to see that the latter can be estimated with a sum of terms of the form
\[
|\partial_{I'}\psi(x)\partial_{I''}(P_{r,s}(x) - v_n(x)))|,
\]
with $I' \in \mathcal{M}(\ell',m)$ and $I'' \in \mathcal{M}(\ell - \ell',m)$. Now $\eqref{estpsi}$, $\eqref{contv}$ and the choice of $P_{r,s}$ yield
\[
|\partial_{I'}\psi(x)\partial_{I''}(P_{r,s}(x) - v_n(x)))| \le c\gamma\eps_{n + 1}
\]
and hence conclude the proof, provided we choose $\gamma$ sufficiently small depending only on $k'$, $m$ and $c$.
\end{proof}

The basic laminate construction as the one of Lemma \ref{lam} has already appeared in the literature in various contexts and for various operators, see for instance \cite[Proposition 3.2]{LINF} and \cite[Lemma 3.3]{COR}. We will now refine it by showing that the map $V_\alpha$ can be chosen to take values in $B_\alpha(a) \cup B_{\alpha}(b)$ instead of $B_\alpha([a,b])$. Closely related results appeared in \cite[Proposition 3.3-3.4]{KIRK} and \cite[Lemma 2.1]{AFSZ}, when studying laminations for the $\curl$ operator in the space of symmetric matrices. 

\begin{prop}\label{lam+}
Let $\Omega\subset \R^m$ be an open and bounded set. Let $a,b \in \R^n$, $b-a  = c \in \Lambda_{\A}$ and $e = \lambda a + (1-\lambda) b$, for some $\lambda \in (0,1)$. Fix any element $q_e \in \POL(k',m)^{n'}$ with the property that $\mathcal{B}(q_e) = e$ everywhere in $\R^m$. Then, for all $\beta > 0$, there exists a map $V_\beta \in W^{k',\infty} \cap C^{k' -1}(\overline{\Omega},\R^{n'})$ such that
\begin{enumerate}
\item\label{f02} the $W^{k',\infty} \cap C^{k' -1}$ norm of $V_\beta$ only depends on $\diam(\Omega),|a|,|b|$ and $|D^{k'}q_e|$;
\item\label{f2} $V_\beta = q_e$, together with all its derivatives of order $\ell < k'$, on $\partial\Omega$;
\item\label{ff2} every component of $V_\beta$ is piecewise a polynomial of degree $k'$;
\item\label{fff2} $\|V_\beta -q_e\|_{C^{k' - 1}(\overline{\Omega})} \le \beta$;
\item\label{ss2} if $v_{\beta}(x) \doteq \mathcal{B}(V_\beta)$, $|\{x \in \Omega: v_\beta(x) \in B_\beta(a)\}| = \lambda|\Omega|$ and $|\{x \in \Omega: v_\beta \in B_\beta(b)\}| = (1-\lambda)|\Omega|$.
\end{enumerate}
\end{prop}

\begin{proof}
Fix $0 < \beta \le \frac{1}{2}|a-b|$ and $0 <\sigma < \min\left\{\frac{\beta}{2|a-b|},\beta\right\}$. We inductively construct a sequence of maps $\{V_n\}_n$ that in the limit will give us a map $V_{\beta,\sigma}$. Let $v_{\beta,\sigma}\doteq \mathcal{B}(V_{\beta,\sigma})$. $V_{\beta,\sigma}$ will have all the required properties, except for \eqref{ss2} that will be replaced by:
\begin{equation}\label{all}
|\Omega| = |\{x \in \Omega: v_{\beta,\sigma} \in B_\beta(a)\}\cup \{x \in \Omega: v_{\beta,\sigma} \in B_\beta(b)\}|
\end{equation}
and
\begin{equation}\label{almost}
|\{x \in \Omega: v_{\beta,\sigma} \in B_\beta(a)\}| \ge (1-\sigma)\lambda|\Omega| \text{ and } |\{x \in \Omega: v_{\beta,\sigma} \in B_\beta(b)\}| \ge (1-\sigma)(1-\lambda)|\Omega|.
\end{equation}
We will deal with \eqref{ss2} in a second moment. 
\\
\\
\;\fbox{Step 1: the inductive setup:}
\\
\\
At step $0$, we choose $V_0 = V_{\alpha}$ for $\alpha = \frac{\sigma}{2} < \frac{\beta}{2}$  and $\Omega_0 \doteq \Omega_2^\alpha$ as in Lemma \ref{lam}. By Lemma \ref{lam}, $\Omega_0$ is essentially open in $\Omega$. Define $\eps_{n} \doteq \frac{\sigma}{2^{n + 2}}$. Suppose we are given a map $V_n \in W^{k',\infty}$ whose components are piecewise polynomials of degree $k'$ that satisfies the following properties
\begin{enumerate}[	$(a)$]
\item\label{ind10} $V_n = q_e$ together with all of its derivatives of order $\ell < k'$ on $\partial\Omega$;
\item\label{ind16} let $v_n \doteq \mathcal{B}(V_n)$. There exists $\Omega_n$, essentially open in $\Omega$, with $|\Omega_n| \le \eps_n|\Omega|$ and such that $$\Omega_n \supseteq \{x: v_{n} \notin B_{\sum_{j}^{n}\eps_j}(a)\cup B_{\sum_{j}^{n}\eps_j}(b)\};$$
\item\label{ind18} $v_{n}(x) \in B_{\sum_{j}^{n}\eps_j}([a,b])$.
\end{enumerate}
We claim it is possible to find a new map $V_{n + 1}\in W^{k',\infty}\cap C^{k'-1}$ whose components are piecewise polynomials of degree $k'$ and with $\|V_{n + 1}\|_{W^{k',\infty}\cap C^{k'-1}} \le \max\{\|V_n\|_{W^{k',\infty}\cap C^{k'-1}},L\}$, where $L$ only depends on $|a|,|b|$ and $|D^{(k)}q_e|$, and fulfilling the following properties:
\begin{enumerate}[	$(A)$]
\item\label{ind20} $V_{n+1} = q_e$, together with all of its derivatives of order $\ell < k'$, on $\partial\Omega$;
\item\label{ind26} let $v_{n + 1} \doteq \mathcal{B}(V_{n + 1})$. There exists $\Omega_{n +1}$, essentially open in $\Omega$, with $|\Omega_{n +1}| \le \eps_{n + 1}|\Omega|$ and  such that $$\Omega_{n+1} \supseteq \{x: v_{n+1} \notin B_{\sum_{j}^{n+1}\eps_j}(a)\cup B_{\sum_{j}^{n+1}\eps_j}(b)\};$$
\item\label{ind28} $v_{n+1}(x) \in B_{\sum_{j}^{n+1}\eps_j}([a,b])$.
\item\label{ind22} $V_{n + 1} = V_n$ on $\Omega_n^c$;
\item\label{ind25} $\|V_{n + 1} - V_n\|_{C^{k' -1}} \le \eps_{n + 1}$;
\end{enumerate}

Suppose for a moment the claim holds. First, Lemma \ref{lam} tells us that $V_0$ satisfies \eqref{ind10}-\eqref{ind16}-\eqref{ind18} for $n = 0$. By \eqref{ind25}, we can define the $C^{k'-1}$ limit
\[
V_{\beta,\sigma} = \lim_n V_n.
\]
Moreover, we have that $\|v_n\|_{L^\infty}$ is equibounded and, by the strong convergence of $V_n$ in $L^\infty$, we infer the weak-$*$ convergence in $L^\infty$ of $v_n$ to $v_{\beta,\sigma} = \mathcal{B}(V_{\beta,\sigma})$. Since $V_n$ and $V_{n + 1}$ differ only on $\Omega_n$ and $|\Omega_n| \to 0$, we see that $V_n$ and $v_n$ converge in measure to $V_{\beta,\sigma}$ and $v_{\beta,\sigma}$, respectively. Now it is easy to deduce from the properties of $V_n$ and $V_{n + 1}$ that $v_{\beta,\sigma}$ and $V_{\beta,\sigma}$ enjoys properties \eqref{f02}-\eqref{f2}-\eqref{ff2}-\eqref{fff2} listed in the statement of the proposition together with \eqref{all}-\eqref{almost}. We now prove the inductive step.
\\
\\
\;\fbox{Step 2: the inductive step.}
\\
\\
Suppose we are given $V_n$ and $\Omega_n$ as above. Split $\Omega_n = \bigcup_q \Omega'_q$, with $\Omega'_q$ open, in such a way that on $\Omega'_q$, every component of $V_n$ is a polynomial of order $k'$. We modify $V_n$ on $\Omega'_q$ in the following way. By \eqref{ind18}, we know that
\[
v_{n}(x) \in B_{\sum_{j}^{n}\eps_j}([a,b]),\quad \forall x \in \Omega,
\]
but from the definition of $\Omega_n$ we also know that
\begin{equation}\label{nonballs}
v_n(x) \in B_{\sum_j^n\eps_j}([a,b])\setminus B_{\sum_{j}^n\eps_j}(a)\cup B_{\sum_{j}^n\eps_j}(b), \quad\forall x \in \Omega'_q.
\end{equation}
Observe that $v_n(x)$ is constant on $\Omega'_q$, since $V_n$ is a vector of polynomials of order $k'$ there. We will then call $e'_q \doteq v_n(x)$. We infer from \eqref{nonballs} that there exists $h_q$ with $|h_q| < \sum_{j = 1}^{n}\eps_j$ and $\mu_q \in (0,1)$ such that
\[
e'_q = h_q + \mu_q a + (1-\mu_q)b = \mu_q (a + h_q) + (1-\mu_q)(b + h_q).
\]
We use Lemma \ref{lam} with $a + h_q, b+ h_q,e'_q,\mu_q, P_q$ instead of $a,b,e,\lambda, q_e$, where $P_q$ is the unique element of $\POL(k',m)^{n'}$ that extends $V_n|_{\Omega_q'}$, to find a map $V_{\rho,q}$ with the properties listed in the statement of Lemma \ref{lam}, for any $0 < \rho < \eps_{n + 1}$. We then replace $V_n$ on $\Omega'_q$ by $V_{\rho,q}$. Call $V_{n + 1}$ the map that coincides with $V_n$ outside of $\Omega_n$ and is defined as $V_{\rho,q}$ in $\Omega'_q$. Notice that we can check the inductive step separately on each subdomain $\Omega'_q$. The fact that 
\[
\|V_{n + 1}\|_{W^{k',\infty}} \le \max\{\|V_n\|_{W^{k',\infty}},L\}
\]
stems from the definition of $V_{n + 1}$ and property \eqref{f0} of $V_{\rho,q}$ stated in Lemma \ref{lam}. Furthermore, \eqref{ind20}-\eqref{ind22} are immediate by construction and \eqref{f1} of Lemma \ref{lam}. \eqref{ind25} is a consequence of the choice $\rho < \eps_{n + 1}$ and \eqref{fff1} of Lemma \ref{lam}. Using \eqref{ss1} of Lemma \ref{lam} and the estimates $\rho < \eps_{n+ 1}$ and $|h_q| < \sum_{j = 1}^n\eps_j$, \eqref{ind28} also follows. Finally, exploiting again the estimates on $|h_q|$ and $\rho$, we also have \eqref{ind26}, by \eqref{t1}-\eqref{ff1} of Lemma \ref{lam}. This concludes the proof of the inductive step.
\\
\\
\fbox{Step 3: proof of \eqref{ss2}.}
\\
\\
This step is analogous to the same step of Lemma \cite[Lemma 2.1]{AFSZ} and we repeat it for the convenience of the reader. Up to now, we have found a map $V_{\beta,\sigma}$ with properties \eqref{f02}-\eqref{f2}-\eqref{ff2}-\eqref{fff2} of the statement of the proposition and with \eqref{ss2} replaced by \eqref{all}-\eqref{almost}, namely:
\[
|\Omega| = |\{x \in \Omega: v_{\beta,\sigma}(x) \in B_\beta(a)\}\cup\{x \in \Omega: v_{\beta,\sigma}(x) \in B_\beta(b)\}|,
\]
and
\[
|\{x \in \Omega: v_{\beta,\sigma} \in B_\beta(a)\}| \ge (1-\sigma)\lambda|\Omega| \text{ and } |\{x \in \Omega: v_{\beta,\sigma} \in B_\beta(b)\}| \ge (1-\sigma)(1-\lambda)|\Omega|.
\]
Since the inductive statement worked for any domain $\Omega$, we now work on a cube\footnote{In fact, any open set $Q$ with $|\partial Q| =0$ would serve for our purpose.} $Q \subset \R^m$ instead of $\Omega$, and we come back to the general bounded open set $\Omega$ of the statement of the proposition later on. We can suppose, without loss of generality that
\[
\lambda|Q| > |\{x \in Q: v_{\beta,\sigma} \in B_\beta(a)\}| \ge (1-\sigma)\lambda|Q|.
\]
Now choose any $s$ such that $ \sigma < s < \min\left\{\frac{\beta}{2|a-b|},1-\lambda\right\}$ and set
\[
a' \doteq a + s(b-a).
\]
Let $\mu = \frac{\lambda}{1 - s} > \lambda$ and write
\[
e = \mu a' + (1-\mu)b.
\]
Since $s < 1 -\lambda$, $\mu \in (0,1)$. We can repeat the previous steps of the proof with $a',b,\lambda$ and $q_e$ in place of $a,b,\mu$ and $q_e$, to obtain a map $V'_{\beta,\sigma}$ with properties \eqref{f02}-\eqref{f2}-\eqref{ff2}-\eqref{fff2} of the statement of the Proposition and with \eqref{ss2} replaced by
\begin{equation}\label{all1}
|Q| = |\{x \in Q: v'_{\beta,\sigma}(x) \in B_\frac{\beta}{2}(a')\}\cup\{x \in Q: v_{\beta,\sigma}(x) \in B_\frac{\beta}{2}(b)\}|,
\end{equation}
and
\begin{equation}\label{almost1}
|\{x \in Q: v'_{\beta,\sigma}(x) \in B_\frac{\beta}{2}(a')\}| \ge (1-\sigma)\mu|Q| \text{ and } |\{x \in Q: v'_{\beta,\sigma}(x) \in B_\frac{\beta}{2}(b)\}| \ge (1-\sigma)(1-\mu)|Q|.
\end{equation}
Here, as usual, $v'_{\beta,\sigma} = B(V'_{\beta,\sigma})$. Since $s < \frac{\beta}{2|a-b|}$, we see that
\[
B_\frac{\beta}{2}(a') \subset B_{\beta}(a),
\]
and hence $\eqref{almost1}$ implies
\begin{equation}\label{almost2}
|\{x \in Q: v'_{\beta,\sigma}(x) \in B_\beta(a)\}| \ge (1-\sigma)\mu|Q| \text{ and } |\{x \in Q: v'_{\beta,\sigma}(x) \in B_\beta(b)\}| \ge (1-\sigma)(1-\mu)|Q|.
\end{equation}
We now come back to the domain $\Omega$ of the statement of the proposition. We split $\Omega$ into two open sets $\Omega_1$ and $\Omega_2$ with $|\Omega_1| = t|\Omega|$, $|\Omega_2| =(1-t)|\Omega|$, $t \in (0,1)$ to be fixed. We subdivide $\Omega_1$ in cubes and fill it with rescaled and translated copies of $V_{\beta,\sigma}$ of the form
\[
V_{\beta,\sigma,r,x_0}(x) \doteq r^{k'}V_{\beta,\sigma}\left(\frac{x- x_0}{r}\right),
\]
and $\Omega_2$ with rescaled and translated copies of $V'_{\beta,\sigma}$ of the same form. The map $V_\beta$ is exactly given by the resulting map, for the correct choice of $t$. Indeed, it is simple to see that $V_\beta$ inherits properties \eqref{f02}-\eqref{f2}-\eqref{ff2}-\eqref{fff2} of the Lemma, and also \eqref{all}-\eqref{all1}, in the sense that
\[
|\Omega| = |\{x \in \Omega: v_{\beta}(x) \in B_\beta(a)\}\cup\{x \in \Omega: v_{\beta}(x) \in B_\beta(b)\}|.
\]
Notice that, by our choice $\beta < \frac{1}{2}|a-b|$, the sets $\{x \in \Omega: v_{\beta}(x) \in B_\beta(a)\}$ and $\{x \in \Omega: v_{\beta}(x) \in B_\beta(b)\}$ are disjoint, thus it suffices to check that there exists $t \in (0,1)$ such that
\[
|\{x \in \Omega: v_{\beta}(x) \in B_\beta(a)\}| = \lambda|\Omega|
\]
to conclude the proof. To see the latter, we write
\begin{align*}
|\{x \in \Omega: v_\beta(x) \in B_\beta(a)\}| &= |\{x \in \Omega_1: v_{\beta}(x) \in B_\beta(a)\}| + |\{x \in \Omega_2: v_\beta(x) \in B_\beta(a)\}|\\
& = |\{x \in Q: v_{\beta,\sigma}(x) \in B_\beta(a)\}|\frac{|\Omega_1|}{|Q|} + |\{x \in Q: v_{\beta,\sigma}'(x) \in B_\beta(a)\}|\frac{|\Omega_2|}{|Q|}\\
& = t|\{x \in Q: v_{\beta,\sigma}(x) \in B_\beta(a)\}|\frac{|\Omega|}{|Q|} + (1-t)|\{x \in Q: v_{\beta,\sigma}' \in B_\beta(a)\}|\frac{|\Omega|}{|Q|}.
\end{align*}
Since $\sigma < s$, $\mu = \frac{\lambda}{1-s}$,
\[
|\{x \in Q: v_{\beta,\sigma}(x) \in B_\beta(a)\}| < \lambda|Q|\text{ and }|\{x \in Q: v'_{\beta,\sigma}(a) \in B_\beta(a)\}| \ge (1-\sigma)\mu|Q| > \lambda |Q|,
\]
it is then clear that there exists $t \in (0,1)$ such that
\[
t|\{x \in Q: v_{\beta,\sigma}(x) \in B_\beta(a)\}| + (1-t)|\{x \in Q: v'_{\beta,\sigma}(a) \in B_\beta(a)\}| = \lambda|Q|.
\]
This choice of $t$ fixes $V_\beta$ and concludes the proof.
\end{proof}

It is convenient to introduce some measure theoretic concept alongside with the simple laminates construction, compare \cite[Section 2]{SMVS}, \cite[Introduction]{KIRK}. For instance, given $a,b$ as in Lemma \ref{lam}, we consider\footnote{The measure we associate is the so-called \emph{Young measure} generated by the sequence of maps defined in Lemma \ref{lam}. We will only use particular Young measures, namely laminates, and hence we will not introduce them in full generality. For a comprehensive introduction, see for instance \cite[Chapter 3]{DMU}.}
\[
\nu = \lambda\delta_{a} + (1-\lambda)\delta_b.
\]
Now, after having split the barycentre $e$ into $a$ and $b$ as $e = \lambda a + (1-\lambda)b$, one may split $b$ as $b = \mu A + (1-\mu)B$, for $\mu \in (0,1)$ and $B - A \in \Lambda_\mathcal{A}$. After this operation, we consider the new measure
\[
\nu' = \lambda\delta_{a} + (1-\lambda)\mu\delta_A + (1-\lambda)(1-\mu)\delta_B.
\]
Notice that the barycentre of $\nu'$ is the same as the one of $\nu$. Generalizing this simple example, we give the following:

\begin{definition} \label{d:laminates_finite}
Let $\nu,\nu' \in \mathcal{P}(U)$, $U \subset \R^{n}$ open. Let $\nu = \sum_{i= 1}^r\lambda_i\delta_{a_i}$. We say that $\nu'$ can be obtained via \emph{elementary splitting from }$\nu$ if for some $i \in \{1,\dots,r\}$, there exist $b,c \in U$, $\lambda \in [0,1]$ such that
\[
b-c \in \Lambda_\A ,\quad [b,c] \subset U, \quad a_i = sb + (1-s)c,
\]
for some $s \in (0,1)$ and
\[
\nu' = \nu +\lambda\lambda_i(-\delta_{a_i} + s\delta_b + (1-s)\delta_c).
\]
A measure $\nu = \sum_{i= 1}^r\lambda_i\delta_{a_i}\in \mathcal{P}(U)$ is called an $\A$-\emph{laminate of finite order} if there exists a finite number of measures $\nu_1,\dots,\nu_{r'} \in \mathcal{P}(U)$ such that
\[
\nu_1 = \delta_X,\quad \nu_{r'} = \nu
\]
and $\nu_{j + 1}$ can be obtained via elementary splitting from $\nu_j$, for every $j\in \{1,\dots,N-1\}$.
\end{definition}

Using the definition of $\mathcal{A}$-laminate of finite order and a simple iterative procedure that exploits Proposition \ref{lam+} at every splitting, one may prove the following result. We refer the interested reader to \cite[Lemma 3.2]{SMVS} for a proof in the case $\mathcal{A} = \curl$.

\begin{prop}\label{ind}
Let $\nu = \sum_{i = 1}^r\lambda_i\delta_{a_i} \in \mathcal{P}(U)$ be an $\mathcal{A}$-laminate of finite order, and let $e = \bar \nu$. Fix any element $q_e \in \POL(k',m)^{n'}$ with the property that $\mathcal{B}(q_e) = e$ everywhere in $\R^m$. Then, given an open set $\Omega$, for every $\eps >0$ there exists $V_\eps \in  W^{k',\infty} \cap C^{k' -1}(\overline{\Omega},\R^{n'})$ enjoying the following properties:

\begin{enumerate}
\item\label{flof0} the $W^{k',\infty} \cap C^{k' -1}$ norm of $V_\eps$ only depends on $\diam(\Omega),\max_i|a_i|$ and $|D^{k'}q_e|$;
\item\label{flof} $V_\eps = q_e$, together with all its derivatives of order $\ell < k'$, on $\partial\Omega$;
\item\label{fflof} Every component of $V_\eps$ is piecewise a polynomial of degree $k'$;
\item\label{ffflof} $\|V_\eps -q_e\|_{C^{k' - 1}(\overline{\Omega})} \le \eps$;
\item\label{sslof} if $v_{\eps}(x) \doteq \mathcal{B}(V_\eps)(x)$, then $|\{x \in \Omega: v_{\eps}(x) \in B_{\eps}(a_i)\}| = \lambda_i|\Omega|,\forall i \in \{1,\dots, r\}$.
\end{enumerate}
\end{prop}

\subsection{Laminates} \label{section:laminates}
In this section we give the definition of $\A$-\emph{laminate}. In \cite[Section 4]{KIRK}, Kirchheim develops all the useful tools concerning $\A$-laminates, thus extending \cite[Section 2]{SMVS} from the case $\mathcal{A} = \curl$ to the case of general linear differential operators. In this subsection, we simply recall the definitions and the results of \cite{KIRK}. Let us point out that in \cite{KIRK} the notation $\mathcal{D}$ is used instead of $\Lambda_\A$ and the name $\mathcal{D}$-\emph{prelaminates} is used instead of $\mathcal{A}$-\emph{laminates of finite order}.

\begin{definition}
Let $O \subset \R^{ n}$ be an open set. We say that $f: O \to \R$ is $\Lambda_\A$-convex in $O$ if $f$ is convex on every $\Lambda_\A$ segment contained in $O$, i.e.
$$ f(\lambda a + (1- \lambda)b ) \leq \lambda f(a) + (1- \lambda) f(b),$$
for any $a,b \in \R^{n}$ such that $a-b \in \Lambda_\A$. If $f$ is $\Lambda_\A$-convex in $\R^n$, we will simply say that $f$ is $\Lambda_\A$-convex.
\end{definition}

\begin{definition}
 Let $E \subset \R^{n}$. We say that $\nu \in \mathcal{P}(E)$ is an $\A$-laminate if 
\begin{equation} \label{jensen}
\int_{\R^{n\times m}}f(X)d\nu \ge f\left(\int_{\R^{n}}Xd\nu\right) = f(\bar \nu),
\end{equation}
for every $\Lambda_\A$-convex function $f$ in $\R^{n}$. We define $$ \mathcal{P}^{ \Lambda_\A}(K)\doteq \{\nu \in \mathcal{P}(K): \nu \text{ is an } \A\text{-laminate} \}.$$
\end{definition}

We give now the definition of $\Lambda_\A$-convex hull of a compact or open subset of $\R^n$. In the case $\mathcal{A} = \curl$, this is the so called rank-one convex hull, $E^{rc}$, compare \cite[Section 6]{SMVS}.

\begin{definition}
Let $K \subset \R^n$ be a compact set. We define the $\Lambda_\A$ convex hull $K^{\Lambda_\A}$ as the set of 
\[
K^{\Lambda_\A} \doteq \{X: X \text{ is the barycenter of a $\A$-laminate $\nu$ in $K$}\},
\]
For an open set $U$,
\[
U^{\Lambda_\A} \doteq \bigcup_{K\subset U: K \text{ compact}}K^{\Lambda_\A}.
\]
\end{definition}
We collect in the next proposition some useful properties of the objects we just introduced:

\begin{prop} \label{p:lambdawavecone}
The following hold:
\begin{enumerate}
\item\label{firstprop} For any compact set $K \subset \R^n$, $$K^{\Lambda_\A} = \{X: f(X) \leq 0, \text{ for every } \ \Lambda_\A \text{-convex $f$ with } \max_{Y\in K}f(Y)\le 0\};$$
\item\label{secondprop} If $U \subset \R^n$ is open, then $U^{\Lambda_\A}$ is open;
\item\label{thirdprop} Let $O \subset \R^n$ be an open and bounded, and let $f:O\to\R$ be $\Lambda_\A$-convex. Then $f$ is locally Lipschitz.
\end{enumerate}
\end{prop}

For the proof of  \eqref{firstprop} we refer the reader to \cite[Corollary 4.11]{KIRK}. \eqref{secondprop} follows from the simple fact that the translation of a laminate is still a laminate. Finally, the proof of \eqref{thirdprop} can be found in \cite[Lemma 2.3]{KKR}.
\\
\\
Notice that if $\nu$ is a $\A$-\emph{laminate of finite order}, then \eqref{jensen} holds for every $\Lambda_\A$-convex function $f$. Since every $\Lambda_\A$-convex function is locally Lipschitz continuous, \eqref{jensen} also holds for every weak-$*$ limit of sequences $\{\nu_n\}_n$ of $\mathcal{A}$-laminates of finite order supported in a fixed bounded open set. Therefore, the weak-$*$ closure of the space of $\mathcal{A}$-laminates of finite order is contained in the space of $\mathcal{A}$-laminates. M\"uller and \v Sver\'ak actually managed to prove the converse in the case of the wave cone induced by the operator $\A =\curl$, compare \cite[Theorem 2.1]{SMVS}. \cite[Theorem 4.12]{KIRK} extends this result to the case of general operators:

\begin{theorem}\label{ann}
Let $K \subset \R^n$ be a compact set and let $\nu \in \mathcal{P}^{\Lambda_\A}(K)$. Let $U$ be an open set such that $K^{\Lambda_\A} \subset U$. Then there exists a sequence $\{\nu_j\}_j \subset \mathcal{P}(U)$ of laminates of finite order  such that $\overline{\nu_j} = \overline{\nu}$ for each $j$ and $\{\nu_j\}_j$ converges weakly-$*$ to $\nu$ in the sense of measures.
\end{theorem}

\subsection{In-approximations and exact solutions}\label{inappexsol}

In this subsection we exploit the theory developed in Section \ref{section:simplelaminates} and Section \ref{section:laminates} to construct solutions of \eqref{AINC}, and in particular we prove Theorem \ref{rcexact}. We start with the following preliminary result.

\begin{prop}\label{usefulprop}
Let $U \subset \R^n$ and $\Omega\subset \R^m$ be open and bounded sets and let $W\in W^{k',\infty}\cap C^{k'-1}(\overline{\Omega},\R^{n'})$ be a map whose components are piecewise polynomials of order $k'$ such that $$\mathcal{B} (W)  \in U^{\Lambda_\A} \text{ in } \Omega.$$ Then, for every $\delta > 0$, there exists a map $V_\delta \in W^{k',\infty} \cap C^{k' -1}(\overline{\Omega},\R^{n'})$ whose components are piecewise polynomials of order $k'$ with the following properties:
\begin{enumerate}
\item the $W^{k',\infty} \cap C^{k' -1}$ norm of $V_\delta$ only depends on $\diam(\Omega)$, $\diam(U)$ and $\|W\|_{C^{k'}}$;
\item $V_\delta = W$, together with all of its derivatives of order $\ell < k'$, on $\partial\Omega$;
\item $\|V_\delta - W\|_{C^{k'-1}} \le \delta$;
\item $v_\delta \doteq \mathcal{B}(V_\delta)$, then  $v_\delta \in U$ a.e. in $\Omega$.
\end{enumerate}
\end{prop}
\begin{proof}
By definition, there exist countably many open and disjoint $\Omega_n$ such that $\Omega = \bigcup_{n}\Omega_n$ and, on $\Omega_n$, $W$ is a vector of polynomials of order $k'$. We work on each $\Omega_n$ separately, and hence fix now $n \in \N$.
\\
\\
By definition, since $e \doteq \mathcal{B}(W|_{\Omega_n}) \in U^{\Lambda_\A}$, there exists a compact set $C\subset U$ such that
\[
e \in C^{\Lambda_\A}.
\]
By Proposition \ref{p:lambdawavecone}, we infer the existence of a $\A$-laminate $\nu$ supported in $C$ with barycentre $e$. Therefore, we can apply Theorem \ref{ann} with $U^{\Lambda_\A}$ instead of $U$. This is possible since $U^{\Lambda_\A}$ is open, see \eqref{secondprop} of Proposition \ref{p:lambdawavecone}. Thus, we can find a  $\A$-laminate of finite order $$\mu = \sum_{i = 1}^r \lambda_i\delta_{a_i}$$ supported in $U^{\Lambda_\A}$, and satisfying
\begin{equation}\label{alp}
\mu(U) \ge \frac{1}{2}\nu(U) = \frac{1}{2}
\end{equation}
the latter coming from the lower semi-continuity on open sets of the total variation of probability measure with respect to the weak-$*$ convergence, see \cite[Theorem 1.40(ii)]{EVG}. We apply Proposition \ref{ind} with $\mu$ and with $q_e \in \POL(k',m)$ chosen to be the unique extension to $\R^m$ of the polynomial $W|_{\Omega_n}$. Hence, fixed $\beta >0$, we know that we can find a map $V^n_\beta\in W^{k',\infty} \cap C^{k' -1}(\overline{\Omega},\R^{n'})$ whose components are piecewise polynomials of degree $k'$ such that:
\begin{enumerate}
\item the $W^{k',\infty} \cap C^{k' -1}$ norm of $V_\beta^n$ is bounded by $\diam(\Omega), \diam(U)$ and $|D^{k'}q_e| \le \|W\|_{W^{k',\infty}}$;
\item $\|V^n_\beta - W\|_{C^{k' -1}(\overline{\Omega_n},\R^{n'})} \le \beta$;
\item $V_\beta^n(x) =W$, together with all of its derivatives of order $\ell < k'$, on $\partial\Omega_n$;
\item if $v_\beta^n \doteq \mathcal{B}(V_\beta^n)$, then $|\{x\in\Omega_n: \dist(v_\beta^n,\{a_1,\dots, a_r\}) \ge \beta\}| = 0$
\item \label{eq:measure}$|\{x \in \Omega_n: \dist(v_\beta^n,a_i) \le \beta\}| = \lambda_i|\Omega_n|,\forall i \in \{1,\dots, r\}$.
\end{enumerate}
By \eqref{alp}, we have
\[
\mu(U) = \sum_{i: a_i \in U}\lambda_i \ge \frac{1}{2}.
\]
We then choose $\beta>0$ so that $B_\beta(a_k) \subset U^{\Lambda_\A}$, $\forall k =1,\dots, r$, and if $a_k \in U$, then also $B_\beta(a_k) \subset U$. This is possible since $U$ is open and we have a finite number of $a_k$. Therefore, \eqref{eq:measure} of the previous list tells us that
\begin{equation}\label{inmeas}
|\{x \in \Omega_n: \mathcal{B}(V_\beta^n) \notin U\}| \le \frac{| \Omega_n|}{2} .
\end{equation}
Now we can define $V_1$ on $\Omega$ by setting $V_1 \doteq V_\beta^n$ on $\Omega_n$. $V_1$ is a map with the required regularity and whose components are all piecewise polynomials of order $k'$. Moreover $V_1=W \text{ on } \partial \Omega$, together with all of its derivatives of order $\ell < k'$,
\begin{equation}\label{props}
\| V_1- W \|_{C^{k' - 1}} \leq \beta,\quad |\{x \in \Omega: \mathcal{B}(V_1) \notin U\}| \le \frac{|\Omega|}{2}.
\end{equation}
Now one iterates this reasoning, considering $V_1$ instead of $W$ and $\{x \in \Omega: \mathcal{B}(V_1) \notin U\}$ instead of $\Omega$. After this step, one gets a map $V_2$ with properties similar to the ones of \eqref{props} with the last one replaced by
\[
|\{x \in \Omega: \mathcal{B}(V_2) \notin U\}| \le \frac{|\{x \in \Omega: \mathcal{B}(V_1) \notin U\}|}{2} \le \frac{|\Omega|}{2^2}.
\]
Iterating this reasoning infinitely many times, one gets a sequence of maps $\{V_q\}_q$ that are easily seen to converge to a map $V_\delta$ with the required properties.
\end{proof}

We now give the definition of $\mathcal{A}$-in-approximation.

\begin{definition}\label{INAPP}
We say that $K \subset \R^n$ \emph{admits a} $\mathcal{A}$-in-approximation if there exists a sequence of open and equibounded sets $U_n \subset \R^{n}$ such that
\begin{equation}\label{lc}
U_n \subset U_{n + 1}^{ {\Lambda_\mathcal{A}}},
\end{equation}
and for every sequence $(X_n)_n$ with $X_n \in U_n$, 
\begin{equation}\label{lp}
\{X_n\}_n\text{ can only have limit points in } K.
\end{equation}
In the sequel, we will simply write $U_n \to K$ for a sequence of open sets having property \eqref{lp}.
\end{definition}

We are now ready to use the proposition above and the concept of $\mathcal{A}$-in-approximation to construct exact solutions of $\eqref{AINC}$.

\begin{theorem}\label{rcexact}
Let $\Omega \subset \R^m$ be an open and bounded set and $K \subset \R^{n}$ be a compact set that admits a $\mathcal{A}$-in-approximation $\{U_n\}_n$. Then, for every $W \in C^{k'}( \overline{\Omega}, \R^{n'})$  such that
\[
\mathcal{B}(W) \in U_1 \text{ in }\Omega,
\]
and for every $\eps>0$, there exists a map $V_\eps\in W^{k',\infty}\cap C^{k'-1}(\overline\Omega,\R^{n'})$ such that:
\begin{enumerate}
\item\label{BOUND}  the $W^{k',\infty}\cap C^{k'-1}$ norm only depends on $\max_n\{\diam(U_n),\diam(K)\}$;
\item\label{BOUNDARY} $V_\eps  = W$,  together with all its derivatives of order $\ell < k'$, on $\partial\Omega$;
\item\label{unif}$\|V_\eps - W\|_{C^{k' -1}(\overline\Omega,\R^{n'})} \leq \eps$;
\item\label{EXA} $\mathcal{B}(V_\eps) \in K, \text{ in }\Omega$.
\end{enumerate}
\end{theorem}

\begin{proof}
Let $W \in C^{k'}(\overline{\Omega},\R^{n'})$. The first thing we do is to replace $W$ with a map $W'$ whose components are piecewise polynomials of degree $k'$ that well approximates $W$ in the $C^{k'}$ norm and has the same boundary datum. To do so, we define for every $j \in \N$
\[
\Omega_j \doteq \{x \in \Omega: \dist(x,\partial \Omega) \ge 2^{-j}\}.
\]
Up to considering $\Omega_{j + j_0}$ for some $j_0 \in \N$, we can assume without loss of generality that $\Omega_0$ is non-empty. Furthermore, again without loss of generality, we can assume that $$0 < \eps < \min_{x \in \overline{\Omega_0}}\dist(\mathcal{B}(W)(x),\partial U_1)$$ and from now on we fix $\eps > 0$. Consider a decreasing sequence $\{c_j\}_j$ of positive numbers such that $c_1 < \frac{\eps}{2}$ and
\begin{equation}\label{vicinov}
\dist(\mathcal{B}(W)(x),\partial U_1) \ge c_{j + 1}, \quad \forall x \in \Omega_{j+1} \setminus \Omega_{j}.
\end{equation}
Applying Lemma \ref{pp} we can find a map $W' \in C^{k'}(\overline{\Omega},\R^{n'})$ whose components are piecewise polynomials of order $k'$ such that
\begin{equation}\label{vicino}
\|W'-W\|_{C^{k'}(\overline{\Omega},\R^{n'})} \le \frac{\eps}{2}, \quad \|W'-W\|_{C^{k'}(\overline{\Omega_{j+1} \setminus \Omega_{j}},\R^{n'})} \le \frac{c_{j+1}}{2}, \forall j,
\end{equation}
and
\begin{equation}\label{bordo}
W'|_{\partial\Omega} = W.
\end{equation}
Now \eqref{vicinov}, \eqref{vicino}, the openness of $U_1$ and the fact that $\mathcal{B}$ is an operator of order $k'$ imply that
\[
\mathcal{B}(W') \in U_1, \text{ in }\Omega.
\]
We now come to the main part of the proof. First, we exploit the property $U_n \subset U_{n+1}^{\Lambda_\A}$ inductively in the following way. We start with $V_1 = W'$. Then, we can apply Proposition \ref{usefulprop} with $U_{i +1}, V_{i}$ instead of $U,W$ and $\delta_{i+ 1}>0$ to find a map $V_{i +1} \in W^{k',\infty} \cap C^{k'-1}(\overline{\Omega},\R^{n'}) $  whose components are piecewise polynomials of degree $k'$ such that:
\begin{enumerate}[\quad\;(i)]
\item\label{equiboud} the $W^{k',\infty} \cap C^{k'-1}$ norm is equibounded by $\max_n\{\diam(U_n),\diam(K)\}$;
\item\label{boudary} $W = V_i = V_{i + 1}$ together with all the derivatives of order $\ell < k'$, on $\partial\Omega$;
\item\label{deltai} $\|V_{i + 1} - V_i\|_{C^{k' -1}} \le \delta_{i+1}$;
\item\label{BVi} $\mathcal{B}(V_{i +1}) \in U_{i +1}$.
\end{enumerate}
The sequence $\{\delta_i\}_i$ is chosen inductively: given $V_i$ and $\delta_i$, we choose suitably $\delta_{i + 1}$, and thus also $V_{i + 1}$ by Proposition \ref{usefulprop}. Using the notation $ \| \cdot \|_{1,i} \doteq \| \cdot \|_{L^1(\Omega_{i},\R^{n'})}$, we find $0 < \eps_i <\min\{2^{-i},\eps_{i -1}\}$ such that
\begin{equation}\label{BVstar}
\|\mathcal{B}(V_i)-\mathcal{B}(V_i)\star \rho_{\eps_i}\|_{1,i} \le \frac{1}{i}.
\end{equation}
In the last equation, we denoted with $\mathcal{B}(V_i)\star \rho_{\eps}$ the mollification of $\mathcal{B}(V_i)$ with the standard even, smooth, compactly supported mollification kernel $\rho_\eps$. Now we choose 
\begin{equation}\label{choicedeltai}
\delta_{i + 1} \doteq \eps\frac{\eps_i}{C2^{i + 1}},
\end{equation}
where $C > 1$ is a universal constant depending only on the choice of the convolution kernel $\rho$. After having made the choice \eqref{choicedeltai}, we continue the iteration. By \eqref{deltai}, we find that $\{V_i\}_i$ is a Cauchy sequence in $C^{k'-1}$. This implies that there exists a limit in the $C^{k'-1}$ topology $V_\eps = \lim_iV_i$. The fact that $V_\eps$ fulfills \eqref{BOUND}-\eqref{BOUNDARY} is an immediate consequence of \eqref{equiboud}-\eqref{boudary}. Furthermore,  \eqref{equiboud} and \eqref{BVi} imply that $\mathcal{B}(V_i)$ is equibounded  in $L^\infty$, and thus the sequence  $\mathcal{B}(V_i)$ is converging weakly-$*$ in $L^\infty$ to $\mathcal{B}(V_\eps)$. We will now prove that
\begin{equation}\label{uinf}
\mathcal{B}(V_\eps) \in K, \text{ in } \Omega.
\end{equation}
The crucial point is that the choice of the sequence $\{\delta_i\}_i$ yields strong $L^1_{\loc}$ convergence of $\mathcal{B}(V_i)$ to $\mathcal{B}(V)$. Once we show this, we can pass to a subsequence that converges pointwise a.e. and use hypothesis \eqref{lp} to conclude \eqref{uinf}. To prove strong $L^1_{\loc}$ convergence, we fix $i_0 \in \N$ and, for all $i > i_0$, we write:
\begin{equation}\label{gradest}
\|\mathcal{B}(V_{i}) - \mathcal{B}(V_\eps)\|_{1,i_0} \le \|\mathcal{B} (V_i) - \mathcal{B}(V_{i})\star\rho_{{\eps}_i}\|_{1,i_0} + \|\mathcal{B}(V_{i})\star\rho_{{\eps}_i} - \mathcal{B}(V_\eps)\star\rho_{{\eps}_i}\|_{1,i_0} + \| \mathcal{B}(V_\eps)-\mathcal{B}(V_\eps)\star\rho_{{\eps}_i}\|_{1,i_0}.
\end{equation}
The first term of the previous sum is converging to $0$ by \eqref{BVstar}, while the latter is converging to $0$ since $\mathcal{B}(V_\eps)$ is an $L^1$ function. It only remains to estimate the middle term of the right hand side of \eqref{gradest}. Since $\mathcal{B} $ is an operator of order $k'$, for the same constant $C  > 1$ appearing in \eqref{choicedeltai}, we can write:
\[
\|\mathcal{B}(V_i)\star\rho_{{\eps}_i} - \mathcal{B}(V_\eps)\star\rho_{{\eps}_i}\|_{1,i_0} \leq C \frac{\|V_i- V_\eps\|_{C^{k'-1}(\overline{\Omega},\R^{n'})}}{\eps_i} \le \frac{C}{\eps_i}\sum_{j = i}^\infty\|V_{j + 1}-V_j\|_{C^{k'-1}} \overset{\eqref{deltai}}{\leq} \frac{C}{\eps_i}\sum_{j = i}^\infty \delta_{j+1}. 
\]
By our choice \eqref{choicedeltai}, we estimate:
\[
\frac{C}{\eps_i}\sum_{j = i}^\infty \delta_{j + 1} \leq \frac{\eps}{2^i}.
\]
Therefore, the right hand side of \eqref{gradest} converges to $0$ as $i \to \infty$. Since $i_0$ was arbitrary and $\Omega_{i_0} \nearrow \Omega$, \eqref{EXA} is proven and it only remains to show \eqref{unif}:
\[
\|V_{\eps} - W\|_{C^{k'-1}(\overline{\Omega},\R^{n'})} \le  \|W' - W\|_{C^{k'-1}} + \|V_\eps-W'\|_{C^{k'-1}} \overset{\eqref{vicino}}{\le} \frac{\eps}{2} + \sum_{i=1}^{\infty}\|V_{i + 1} - V_i\|_{C^{k'-1}} \le \frac{\eps}{2} + \frac{\eps}{2} = \eps.
\]
This concludes the proof.
\end{proof}

\section{The four state problem}\label{fours}

In this section, we study the inclusion

\begin{equation}\label{AINCfour}
\begin{cases}
v(x) \in K \doteq \{a_1,a_2,a_3,a_4\} \subset \R^n,& \text{a.e. in $B_1$,}\\
\mathcal{A}(v) = 0, &\text{ in the sense of distributions},
\end{cases}
\end{equation}
with $v \in L^\infty(B_1,\R^{n'})$, $B_1 \subset \R^m$ being the ball of radius $1$ centred at $0$, and $a_i -a_j \notin \Lambda_\mathcal{A}$ if $i \neq j$. We wish to exploit Theorem \ref{rcexact} to solve \eqref{AINCfour}. In order to do so we need to find a $\mathcal{A}$-in-approximation for $K$. In \cite[Definition 2.6]{FS}, F\"orster and Sz{\'{e}}kelyhidi introduced the notion of large $T_5$ in order to find a solution to \eqref{AINCfour} in the case $\mathcal{A} = \curl$ and five states. We give the analogous definition in our case.

\begin{definition}\label{T4}
Let $S \subset \R^n$ be arbitrary. We say that an ordered set of elements $(a_1,a_2,a_3, a_4)$ \emph{are in $S$-$T_4$ configuration} if there exist $p \in \R^n$, $c_1,c_2,c_3,c_4 \in S \subset \R^n$ and $k_1,k_2,k_3,k_4 \in (1,+\infty)$ such that
\begin{equation}\label{t4}
\begin{cases}
&a_1 = p + k_1c_1\\
&a_2 = p + c_1 + k_2c_2\\
&a_3 = p + c_1 + c_2 + k_3c_3\\
&a_4 = p + c_1 + c_2 + c_3 + k_4c_4\\
&c_1 + c_2 + c_3 + c_4 = 0.
\end{cases}
\end{equation} 
We say that $\{a_1,a_2,a_3,a_4\}$ \emph{form a large $S$-$T_4$ configuration} if there exist $3$ distinct permutations $\sigma_1,\sigma_2,\sigma_3: \{1,2,3,4\} \to \{1,2,3,4\}$ such that the ordered set of vectors $(a_{\sigma_i(1)},a_{\sigma_i(2)},a_{\sigma_i(3)},a_{\sigma_i(4)})$ is in $S$-$T_4$ configuration, i.e. it fulfills
\begin{equation}\label{t4perms}
\begin{cases}
&a_{\sigma_{i}(1)} = p^{\sigma_i} +k^{\sigma_i}_{1}c^{\sigma_i}_{1}\\
&a_{\sigma_{i}(2)} = p^{\sigma_i} + c_1^{\sigma_i} + k^{\sigma_i}_{2}c^{\sigma_i}_{2}\\
&a_{\sigma_i(3)} = p^{\sigma_i} + c_1^{\sigma_i} + c^{\sigma_i}_{2} + k^{\sigma_i}_{3}c^{\sigma_i}_{3}\\
&a_{\sigma_i(4)} = p^{\sigma_i} + c_1^{\sigma_i} + c^{\sigma_i}_{2} + c^{\sigma_i}_{3}+ k^{\sigma_i}_{4}c^{\sigma_i}_{4}\\
&c_1^{\sigma_i} + c^{\sigma_i}_{2} + c^{\sigma_i}_{3}+ c^{\sigma_i}_{4} = 0.
\end{cases}
\end{equation}
for vectors $p^{\sigma_i}$, $k^{\sigma_i}_{\ell}$, $c^{\sigma_i}_{\ell} $, for $1\le i \le 3$, $1 \le \ell \le 4$, and moreover the vectors $c^{\sigma_1}_{\sigma_{1}^{-1}(\ell)},c^{\sigma_2}_{\sigma_{2}^{-1}(\ell)},c^{\sigma_3}_{\sigma_{3}^{-1}(\ell)} \in \Lambda_\A$ are linearly independent for every fixed $\ell \in \{1,2,3,4\}$.
\end{definition}

Definition \ref{T4} becomes meaningful when $S = \Lambda_\A$ for some $\mathcal{A} \in \OP$ with potential $\mathcal{B} \in \op(k',m,n',n)$. In this case, $\Lambda_\A$-$T_4$ configurations are the most studied example of sets without $\Lambda_\A$ that display \emph{flexibility} for approximate solutions. Indeed given a $\Lambda_\A$-$T_4$ configuration $K = \{a_1,a_2,a_3,a_4\}$ as in \eqref{t4} one can find a sequence of equibounded maps $\{u_n\}_n \subset L^\infty$ such that
\[
\begin{cases}
\dist(u_n,K) \to 0, &\text{strongly in }L^1\\
u_n \weak P, &\text{as }n \to \infty,\\
\mathcal{A}(u_n) = 0, &\forall n \in \N,
\end{cases}
\]
where $P \in \POL(k',m)$ is such that $\mathcal{B}(P) = p$ everywhere on $\R^m$. This stems from the fact that $p \in B_r(\{a_1,a_2,a_3,a_4\})^{\Lambda_\A}$ for all $r > 0$ and hence Proposition \ref{usefulprop} applies. For an introduction to $\Lambda_{\curl}$-$T_4$ configurations, that are simply called $T_4$ configurations in the literature, see \cite[Lemma 2.6]{DMU}.
\\
\\
While $\Lambda_\A$-$T_4$ are related to the existence of approximate solutions, large $\Lambda_\A$-$T_4$ configurations yield the existence of $\A$-in approximations of $K$ and hence through Theorem \ref{rcexact} exact solutions. This was noticed first in \cite{FS}. The fact that the existence of a large $\Lambda_\A$-$T_4$ configuration $\{a_1,a_2,a_3,a_4\}$ implies the existence of a $\A$-in-approximation is analogous to the proof of the same fact for the $\curl$ operator given in \cite{FS}, and will be sketched in Subsection \ref{larget4}.
\\
\\
If $S$ is not (a subset of) a cone $\Lambda_\A$, Definition \ref{T4} has a purely algebraic meaning, and we chose to give it in that way for improving the clarity of our exposition. Indeed, in our strategy, we will first find a set $\{a_1,a_2,a_3,a_4\}$ and write it as a large $\R^n$-$T_4$ configuration. At this level, this only means computing the values $p^\sigma, k_\ell^{\sigma_i}, c_\ell^{\sigma_i}$ for which the algebraic condition \eqref{t4} are satisfied for the three permutations $\sigma_1,\sigma_2,\sigma_3$, and the additional requirement on the linear independence of $\left\{c^{\sigma_i}_{\sigma_{i}^{-1}(\ell)}, 1\le i \le 3\right\}$ for all $1\le \ell \le 4$. Note that this is always possible. Subsequently, we find an operator $\A$ for which $\{c_\ell^{\sigma_i}, 1\le i \le 3, 1 \le \ell \le 4\} \subset \Lambda_\A$, thus proving that $\{a_1,a_2,a_3, a_4\}$ is a large $\Lambda_\A$-$T_4$ configuration. More precisely, we construct an operator $\A$ for which
\begin{equation}\label{c}
\{c^{\sigma_i}_\ell: 1\le i \le 3 , 1\le \ell \le 4\} \subset \Lambda_\A
\end{equation}
and
\begin{equation}\label{a-a}
a_i - a_j \notin \Lambda_\A, \quad \forall i \neq j.
\end{equation}
In order to guarantee that $c^{\sigma_1}_{\sigma_{1}^{-1}(\ell)},c^{\sigma_2}_{\sigma_{2}^{-1}(\ell)},c^{\sigma_3}_{\sigma_{3}^{-1}(\ell)}$ are linearly independent, for fixed $\ell \in \{1,2,3,4\}$, we need to have $ n \ge 3$, and hence we fix $n = 3$. We then choose the following vectors:
\begin{equation}\label{aaaa}
a_1 =\left(\begin{array}{c} 0 \\ 0 \\ 0\end{array}\right),\; a_2 =\left(\begin{array}{c} 1 \\ 0 \\ 0\end{array}\right),\; a_3 =\left(\begin{array}{c} 0 \\ 1 \\ 0\end{array}\right),\;a_4 =\left(\begin{array}{c} 0 \\ 0 \\ 1\end{array}\right).
\end{equation}
The three permutations $\sigma_1,\sigma_2,\sigma_3$ and the elements $p^{\sigma_i}$, $k^{\sigma_i}_{\ell}$, $c^{\sigma_i}_{\ell}$, for $1\le \ell \le 4$, $1 \le i \le 3$ for which $\{a_1,a_2,a_3,a_4\}$ form a large $\R^n$-$T_4$ configuration will be given in Subsection \ref{exval}. From now on, we treat all of these values as fixed, explicit values.
\\
\\
Let us now examine requirement \eqref{a-a}. We can rewrite \eqref{a-a} as:
\[
a_i - a_j \notin \Ker(\mathbb{A}(\xi)) = \im(\mathbb{B}(\xi)), \quad \forall \xi \in \R^m,
\]
if, as usual, $\mathcal{B} \in \op(k',m,n',n)$ denotes the potential of $\A$. Heuristically, \eqref{a-a} has more chances to be satisfied once $ \Lambda_\A = \bigcup_{\xi \in \R^m}\im(\mathbb{B}(\xi))$ is chosen as small as possible.  It is therefore natural to ask $n',m < n = 3$, and indeed we fix $n' = 1$ and $m = 2$. Notice anyway that \eqref{c} is asking that $\Lambda_\A$ contains \emph{at least} the $12$ vectors $c_\ell^{\sigma_i}$, and in order to achieve this, we will use our last degree of freedom $k'$. Given the constraints $n' = 1, m=2$, we have that, for $q_1,q_2,q_3 \in \PO(k',2)$ that will be chosen later,
\begin{equation}\label{BP}
\mathbb{B}(\xi) = \left(\begin{array}{c} q_1(\xi)\\ q_2(\xi) \\ q_3(\xi)\end{array}\right).
\end{equation}
We identify the linear application $\mathbb{B}(\xi)$ with its associated matrix. It only remains to deal with \eqref{c}. Each $q_i \in \POL(k',2)$ has $k'+1$ coefficients. \eqref{c} is now equivalent to asking the existence of twelve vectors: $(\xi_1^{\sigma_i},\xi_2^{\sigma_i}), (\xi_3^{\sigma_i},\xi_4^{\sigma_i}), (\xi_5^{\sigma_i},\xi_6^{\sigma_i}),(\xi_7^{\sigma_i},\xi_8^{\sigma_i}) \in \R^2$ , for $1\le i \le 3$, such that
\begin{equation}\label{Pc}
\mathbb{B}((\xi_{2\ell-1}^{\sigma_i}, \xi^{\sigma_i}_{2\ell})) = c^{\sigma_i}_\ell, \quad \forall 1\le i\le 3, \forall 1\le \ell \le 4.
\end{equation}
We randomly generate the vectors $(\xi^{\sigma_i}_{2\ell-1}, \xi^{\sigma_i}_{2\ell}) \in \R^2$, see Subsection \ref{exval} for the explicit values. Recall that also the $12$ vectors $c^{\sigma_i}_\ell \in \R^3$ are fixed, and thus \eqref{Pc} becomes a linear system of $36$ equations that can be solved using the coefficients of $q_1,q_2,q_3$. The right number of variables is therefore $36$, that amounts to ask
\[
3(k' + 1) = 36,
\]
or $k' = 11$. This last choice fixes all the degrees of freedom of $\mathcal{B}\in \op(11,2,1,3)$. Of course, we should now find an operator $\A \in \op(k,2,3,N)$ whose potential is $\mathcal{B}$. We define our candidate $\A$ by writing its symbols $\mathbb{A}(\xi)$. First, we choose
\[
k = 11 \text{ and } N=3,
\]
and set, for all $\xi \in \R^2$,
\begin{equation}\label{Axi}
\mathbb{A}(\xi) \doteq \left(\begin{array}{ccc} 0 & -q_3(\xi)& q_2(\xi)\\ -q_3(\xi) & 0 & q_1(\xi) \\ -q_2(\xi) & q_1(\xi) & 0\end{array}\right).
\end{equation}
In order to apply the convex integration methods of the previous section, we need the operator $\mathbb{A}$ to be of constant rank and balanced. Furthermore, we need to find a way to verify \eqref{a-a}. This will be done in Theorem \ref{check}. First, we collect our set of assumptions in the following:

\begin{prop}\label{computer}
Let $a_1,a_2,a_3,a_4$ be as in \eqref{aaaa}, and $p^{\sigma_i}$, $k^{\sigma_i}_{\ell}$, $c^{\sigma_i}_{\ell}$, $(\xi_{2\ell-1}^{\sigma_i}, \xi^{\sigma_i}_{2\ell})$ for $1\le \ell \le 4$, $1 \le i \le 3$ as in Subsection \ref{exval}. Then, there are unique polynomials $q_1,q_2,q_3 \in \PO(11,2)$ such that \eqref{Pc} is feasible. Furthermore:
\begin{enumerate}
\item\label{comp0} \eqref{t4perms} holds for all $\sigma_i$, $1 \le i \le 3$;
\item\label{comp1} $c^{\sigma_1}_{\sigma_{1}^{-1}(\ell)},c^{\sigma_2}_{\sigma_{2}^{-1}(\ell)},c^{\sigma_3}_{\sigma_{3}^{-1}(\ell)}$ are linearly independent for all fixed $\ell \in \{1,2,3,4\}$;
\item\label{comp3} $q_i$ and $q_j$ have no common zero on $\mathbb{S}^1$, for all $i \neq j$;
\item\label{comp4} $q_i + q_j$ has no common zero with $q_k$ on $\mathbb{S}^1$, for all $i,j,k$ such that $\{i,j,k\} = \{1,2,3\}$;
\end{enumerate}
\end{prop}
\begin{proof}
All of the above checks have been made using Maple 2020 using symbolic calculus and the fractional representation of rational numbers, hence they are formally justified. We will now explain how to perform these computations on a computer in such a way that the result is rigorous, especially \eqref{comp3} and \eqref{comp4}.
\\
\\
We use coordinates $\xi = (x,y)$ in $\R^2$. First, \eqref{comp0}-\eqref{comp1} are simple computations that could be potentially done by hand. Next, one checks that \eqref{Pc} has a solution. Since the solution is unique and the coefficients are particularly lengthy, we do not write them here explicitly. Notice that, since \eqref{Pc} is a linear system with rational entries, the solution is also rational. Thus, every coefficient of $q_1,q_2,q_3$ can be exactly represented as a fraction. Furthermore, using the explicit form of $q_1,q_2,q_3$, one can easily check the following, for all $1\le i,j \le 3$:
\begin{equation}\label{coeff}
\text{the coefficient of $x^{11}$ of $q_i(x,y)$ and of $q_{i}(x,y) + q_j(x,y)$ is non-zero.}
\end{equation}
Now we turn to \eqref{comp3}-\eqref{comp4}. Let us define $r_k(\xi) \doteq q_i(\xi) + q_j(\xi)$,  for $i,j,k$ such that $\{i,j,k\} = \{1,2,3\}$. By homogeneity, we have, for all $i \in \{1,2,3\}$ and $y \neq 0$,
\begin{equation}\label{qiri}
q_i(x,y) = y^{11}q_i\left(\frac{x}{y},1\right), \quad r_i(x,y) = y^{11}r_i\left(\frac{x}{y},1\right).
\end{equation}
Therefore, we can associate to every $q_i$ and $r_i$ a polynomial of one variable, $Q_i(z)$ and $R_i(z)$, defined as
\[
Q_i(z) \doteq q_i(z,1), \quad R_i(z) \doteq r_i(z,1).
\]
By \eqref{coeff} and \eqref{qiri}, zeroes of $q_i$ and $r_i$ are in bijective correspondence with the ones of $Q_i$, $R_i$, in the sense that $q_i(x_0,y_0) = 0$ if and only if $Q_i\left(\frac{x_0}{y_0}\right) = 0$, and analogously for $r_i$. \eqref{comp3}-\eqref{comp4} then become equivalent to the following:
\begin{enumerate}[	(i)]
\item\label{comp3bis} $Q_i$ and $Q_j$ have no common zero, for all $i \neq j$;
\item\label{comp4bis} $R_i$ has no common zero with $Q_i$, for all $i \in \{1,2,3\}$.
\end{enumerate}
Given the explicit forms of the $Q_i$ and $R_i$, there are two ways to check \eqref{comp3bis}-\eqref{comp4bis}. 
\\
\\
The first starts by computing the zeroes of the six polynomials $Q_1,Q_2,Q_3,R_1,R_2,R_3$ numerically. Of course, this does not yield a rigorous proof, but then one only uses the numerical values to find (small) intervals with rational endpoints \emph{around} those numerical zeroes, in such a way to have that the polynomial evaluated at the two endpoints has two different signs. Notice that the evaluation at the endpoints again gives an exact value, as the polynomial is rational and the endpoints have been chosen to be rational. By continuity it follows that a zero of the polynomial lies inside this interval. Now, instead of having \emph{different} zeroes, we may simply try to find \emph{disjoint} intervals, which would suffice to show \eqref{comp3bis}-\eqref{comp4bis}.
\\
\\
The second way to check \eqref{comp3bis}-\eqref{comp4bis} is much quicker in terms of computations, and it is the method we employed. This simply consist in computing the GCD of every couple $Q_i,Q_j$ and $Q_i,R_i$, varying $1\le i\neq j \le 3$. If the GCD of these couples is a constant, then clearly they can have no common zero, and this turns out to be the case in our particular example. Since the polynomials depend only on one variable, we can use the Euclidean algorithm to compute the GCD among the couples of polynomials we are interested in. Using the built-in \emph{gcd} function of Maple 2020, we have checked that $GCD(Q_i,R_i) = GCD(Q_i,Q_j) = 1, \forall 1\le i \neq j \le 3$, and hence \eqref{comp3bis}-\eqref{comp4bis} hold.
\end{proof}

\begin{corollary}\label{c:potential}
Let $a_1,a_2,a_3,a_4$ be as in \eqref{aaaa}, and $p^{\sigma_i}$, $k^{\sigma_i}_{\ell}$, $c^{\sigma_i}_{\ell}$, $(\xi_{2\ell-1}^{\sigma_i}, \xi^{\sigma_i}_{2\ell})$ for $1\le \ell \le 4$, $1 \le i \le 3$ as in Subsection \ref{exval}. Finally, let $q_1,q_2,q_3$ be the only solution of \eqref{Pc} and define the two operators $\mathcal{A}$ as in \eqref{Axi} and $\mathcal{B}$ as in \eqref{BP}. Then:
\begin{itemize}
\item $a_i - a_j \notin \Lambda_\A, \forall 1 \le i < j \le 4$;
\item $\mathcal{B}$ is a potential for $\mathcal{A}$ in the sense of \eqref{pot};
\item $\mathcal{A}$ has constant rank and is balanced;
\end{itemize}
\end{corollary}
\begin{proof}
Fix $1\le i < j \le 4$. To see that $a_i - a_j \notin \Lambda_\A $ we notice that, by \eqref{aaaa}, $a_i - a_j$ either has two zeroes, or it has a zero component while the other two are $1$ and -1. In the first case, i.e. when $a_i - a_j$ has two zeroes, $a_i - a_j \notin \Lambda_\A$ stems from \eqref{comp3} of Proposition \ref{computer}, while in the second, $a_i - a_j \notin \Lambda_\A$ is a consequence of \eqref{comp4} of Proposition \ref{computer}.
\\
\\
We show now that $\mathcal{B}$ is a potential for $\A$. For all $\xi \in \R^2$, it holds
\[
\im(\mathbb{B}(\xi)) \subset \Ker(\mathbb{A}(\xi)),
\]
that shows $\rank(\mathbb{A}(\xi)) \le 2$ since $\mathbb{B}(\xi) \neq 0, \forall \xi \in \R^2\setminus\{0\}$ by \eqref{comp3} of Proposition \ref{computer}. The principal $2\times 2$ minors of $\mathbb{A}(\xi)$ read as $-q_3^2(\xi), q_2^2(\xi), -q_1^2(\xi)$ and thus again by property \eqref{comp3} of Proposition \ref{computer}, we find that $\rank(\mathbb{A}(\xi)) = 2$ for all $\xi \in \R^2\setminus\{0\}$. It also follows that
\[
\im(\mathbb{B}(\xi)) = \Ker(\mathbb{A}(\xi)).
\]
This shows that $\mathcal{A}$ has constant rank and $\mathcal{B}$ is the potential of $\A$ in the sense of \eqref{pot}. Finally, we show that $\A$ is balanced.  By \eqref{Pc} and \eqref{comp1} of Proposition \ref{computer} $\im(\mathbb{B}(\xi))$ contains three linearly independent vectors. Since $\im(\mathbb{B}(\xi)) \subset \Lambda_\A$ for all $\xi \in \R^2$, it follows $\spn(\Lambda_\A) = \R^3$ and the proof is finished.
\end{proof}
Now we can finally prove the main result of this paper, namely the existence of a nontrivial solution of \eqref{AINCfour}.

\begin{theorem}\label{check}
Let $a_1,a_2,a_3,a_4$ be as in \eqref{aaaa}, $q_1,q_2,q_3 \in \PO(11,2)$ be defined by \eqref{Pc} for the values $(\xi_{2\ell-1}^{\sigma_i}, \xi^{\sigma_i}_{2\ell})$ and $c_{\ell}^{\sigma_i}$ given in Subsection \ref{exval}. Finally define the two operators $\mathcal{A}$ as in \eqref{Axi} and $\mathcal{B}$ as in \eqref{BP}. Then, there exists a non-constant solution $v \in L^\infty(B_1,\R^3)$ of
\[
\begin{cases}
v(x) \in K \doteq \{a_1,a_2,a_3,a_4\}, & \text{a.e. in $B_1$,}\\
\mathcal{A}(v) = 0, &\text{in the sense of distributions},\\
a_i - a_j \notin \Lambda_\A, & \forall i \neq j.
\end{cases}
\]
Furthermore, $v$ takes all four values of $K$, and $v$ admits a potential, i.e. there exists $V \in W^{11,\infty}\cap C^{10}(\overline{B_1})$ such that a.e. on $\Omega$
\[
v = \mathcal{B}(V) 
\]
and $V$ coincides with a polynomial of order $11$ on the boundary of $\partial B_1$.
\end{theorem}
\begin{proof}
By \eqref{comp0}-\eqref{comp1} of Proposition \ref{computer}, we find that $\{a_1,a_2,a_3,a_4\}$ form a large $\Lambda_\A$-$T_4$ configuration. Moreover, Corollary \ref{c:potential} yields that $a_i - a_j \notin \Lambda_\A,\forall 1 \le i < j \le 4$. By Theorem \ref{existenceinapp}, there exists a $\A$-in-approximation $\{U_n\}_n$ of $K$. Since by Corollary \ref{c:potential} $\A$ is a balanced and of constant rank with potential $\mathcal{B}$, we are in position to apply Theorem \ref{rcexact}. Now fix any polynomial $r \in \POL(11,2)$ such that $\mathcal{B}(r) \in U_1$ everywhere on $\R^2$. This exists by Proposition \ref{surj}. Using the existence of three linearly independent direction $c_1,c_2,c_3 \in \Lambda_\A$, that is \eqref{comp1} of Proposition \ref{computer}, in combination with Proposition \ref{ind}, it is not difficult to build a map $W \in C^{11}( \overline{B_1})$ such that $W(x) = r(x), \text{ on } \partial B_1$, $\mathcal{B}(W)(x) \in U_1$, $\forall x \in B_1$
\begin{equation}
\spn\{\im(\mathcal{B}(W))\} \text{ is not contained in an affine subspace of $\R^3$ of dimension $\le 2$}.\label{propfi3}
\end{equation}
Now, by Theorem \ref{rcexact} we find a family $V_\eps$ of maps that are equibounded in $W^{11,\infty}\cap C^{10}(\overline{B_1})$ such that $V_\eps = r$ on $\partial B_1$, $\mathcal{B}(V_\eps) \in K$ a.e. and $V_\eps \to W$ as $\eps \to 0^+$. This yields the weak-$*$ convergence of $\mathcal{B}(V_\eps)$ to $\mathcal{B}(W)$ in $L^\infty$, and hence the weak convergence in $L^2$. If, by contradiction, for all $\eps > 0$, $\mathcal{B}(V_\eps)$ belonged to a proper subset of $K$, say to $\{a_1,a_2,a_3\}$, then
\[
\mathcal{B}(V_\eps) \in \co(\{a_1,a_2,a_3\}), \quad \forall \eps >0.
\]
By Mazur Lemma we would find that $\mathcal{B}(W)(x) \in \co(\{a_1,a_2,a_3\})$, for all $x \in B_1$, and this is in contradiction with $\eqref{propfi3}$. This concludes the proof of the Theorem.
\end{proof}

\subsection{Large $\Lambda_\A$-$T_4$ configurations and in-approximations}\label{larget4}

In this subsection, we collect the main results concerning large $\Lambda_\A$-$T_4$ configurations that we used in the previous section. We will always work with the operator $\A$ defined in $\eqref{Axi}$ and use the objects and the notation introduced in the previous section. Most of the theory immediately follows from the results of \cite{FS} with minor modifications. Thus, some proofs will be omitted and precise reference to the corresponding results of \cite{FS} will be provided.  Let us start with the following:

\begin{lemma}\label{IMT}
Let $\{a_1,a_2,a_3,a_4\} \subset \R^3$ be defined as in \eqref{aaaa}, and let $\sigma_i, 1\le i\le 3$ be the three permutations of \eqref{permut} for which the \emph{ordered} sets $(a_{\sigma_i(1)},a_{\sigma_i(2)},a_{\sigma_i(3)},a_{\sigma_i(4)})$ are in $\Lambda_\mathcal{A}$-$T_4$ configuration. Define $A_i \doteq (a_{\sigma_i(1)},a_{\sigma_i(2)},a_{\sigma_i(3)},a_{\sigma_i(4)}) \subset (\R^{3})^4$. Then, the following hold for all $1\le i \le 3$:
\begin{enumerate}
\item\label{intorno} there exists $\eps > 0$ such that all points $X = (x_1,x_2,x_3,x_4) \in B_\eps(A_i)$ are in $\Lambda_\A$-$T_4$ configuration;
\item \label{intornolarge} there exists $\eps > 0$ such that all points $X = (x_1,x_2,x_3,x_4) \in B_\eps(A_1)$ form a large $\Lambda_\A$-$T_4$ configuration with the same permutations \eqref{permut};
\item\label{welldef} write $X = (x_1,x_2,x_3,x_4) \in B_\eps(A_1)$ as in \eqref{t4perms} for vectors $p^{\sigma_i}(X)$, $k^{\sigma_i}_{\ell}(X)$, $c^{\sigma_i}_{\ell}(X)$, for $1\le i \le 3$, $1 \le \ell \le 4$. Then the maps $\Phi^{\sigma_i}: B_\eps(A_1) \to (\R^3)^{ 4}$ defined as $$\Phi^\sigma(X) \doteq (c_1^{\sigma_i}(X),c_2^{\sigma_i}(X),c_3^{\sigma_i}(X),c_4^{\sigma_i}(X))$$ are well-defined and smooth.
 \end{enumerate}
\end{lemma}

Here we cannot argue as in \cite[Lemma 2.4]{FS} since it relies heavily on the characterization of $T_4$ configurations for the $\curl$ operator shown in \cite{LSR}, and hence we explain the proof in detail.

\begin{proof}
Clearly, \eqref{intornolarge} follows from Definition \ref{T4} and \eqref{intorno}. To show \eqref{intorno}-\eqref{welldef}, we rely on the implicit function theorem and the inverse function theorem. Throughout the proof, we fix  $1\le i \le 3$. First, we use the implicit function theorem to show that the map $\Psi$ defined as
\[
\Psi_i(\xi_1,\xi_2,\xi_3,\xi_4,\xi_5,\xi_6,\xi_7,\xi_8) = v(\xi_1,\xi_2)+ v(\xi_3,\xi_4) + v(\xi_5,\xi_6) + v(\xi_7,\xi_8)
\]
has a non-degenerate set of zeroes in a neighbourhood of $\Xi^{\sigma_i} \doteq (\xi_1^{\sigma_i},\xi_2^{\sigma_i},\xi_3^{\sigma_i},\xi_4^{\sigma_i},\xi_5^{\sigma_i},\xi_6^{\sigma_i},\xi_7^{\sigma_i},\xi_8^{\sigma_i})$, for the explicit values  $\xi_\ell^{\sigma_i}$ of Subsection \ref{exval}. To do so, it is sufficient to compute the determinants of the matrix
\[
(\partial_2v(\xi_5^{\sigma_i},\xi_6^{\sigma_i})| \partial_1v(\xi_7^{\sigma_i},\xi_8^{\sigma_i})|\partial_2v(\xi_7^{\sigma_i},\xi_8^{\sigma_i}))
\]
and see that they are all non-zero. This can be easily checked with the use of a computer. We infer that in a small neighbourhood $D$ of $(\xi_1^{\sigma_i},\xi_2^{\sigma_i},\xi_3^{\sigma_i},\xi_4^{\sigma_i},\xi_5^{\sigma_i},\xi_6^{\sigma_i},\xi_7^{\sigma_i},\xi_8^{\sigma_i})$, $$\Psi_i(\xi_1,\xi_2,\xi_3,\xi_4,\xi_5,\xi_6,\xi_7,\xi_8) = 0$$ if and only if
\[
\xi_\ell = \xi_\ell(\xi_1,\xi_2,\xi_3,\xi_4,\xi_5), \quad \forall 6\le \ell \le 8,
\]
for some smooth maps $\xi_\ell$ defined in a small neighbourhood of $(\xi_1^{\sigma_i},\xi_2^{\sigma_i},\xi_3^{\sigma_i},\xi_4^{\sigma_i},\xi_5^{\sigma_i})$. Define $D' \doteq \pi(D)$, where $\pi$ is the projection on the first $5$ coordinates. Observe that $D'$ is open. Now finally define
\[
F_i: \R^3\times \R^4 \times D' \to \R^{12},
\]
as
\[
F_i(p,k_1,k_2,k_3,k_4,\bar \xi) \doteq \left(\begin{array}{c} p + k_1v(\xi_1,\xi_2)\\ p + v(\xi_1,\xi_2) + k_2v(\xi_3,\xi_4) \\ p + v(\xi_1,\xi_2) + v(\xi_3,\xi_4) + k_3v(\xi_5,\xi_6(\bar \xi))\\
p + v(\xi_1,\xi_2) + v(\xi_3,\xi_4) + v(\xi_5,\xi_6(\bar \xi)) + k_4v(\xi_7(\bar \xi),\xi_8(\bar \xi))\end{array}\right),
\]
where we used the short hand notation $\bar \xi \doteq (\xi_1,\xi_2,\xi_3,\xi_4,\xi_5)$. Now we apply the inverse function theorem at the point defined by the exact values of Subsection \ref{exval}:
\[
(p^{\sigma_i}, k_1^{\sigma_i},k_2^{\sigma_i},k_3^{\sigma_i},k_4^{\sigma_i}, \xi_1^{\sigma_i},\xi_2^{\sigma_i},\xi_3^{\sigma_i},\xi_4^{\sigma_i},\xi_5^{\sigma_i}),
\]
Notice that the derivatives of $\xi_\ell(\bar \xi)$ at the point $(\xi_1^{\sigma_i},\xi_2^{\sigma_i},\xi_3^{\sigma_i},\xi_4^{\sigma_i},\xi_5^{\sigma_i})$ are explicitly provided by the implicit function theorem applied in the first part of this proof. Again with the help of a computer, one can check that
\[
\det(DF_i(p^{\sigma_i}, k_1^{\sigma_i},k_2^{\sigma_i},k_3^{\sigma_i},k_4^{\sigma_i}, \xi_1^{\sigma_i},\xi_2^{\sigma_i},\xi_3^{\sigma_i},\xi_4^{\sigma_i},\xi_5^{\sigma_i})) \neq 0,
\]
and hence that $F_i$ is a diffeomorphism around $(p^{\sigma_i}, k_1^{\sigma_i},k_2^{\sigma_i},k_3^{\sigma_i},k_4^{\sigma_i}, \xi_1^{\sigma_i},\xi_2^{\sigma_i},\xi_3^{\sigma_i},\xi_4^{\sigma_i},\xi_5^{\sigma_i})$. This shows \eqref{intorno}-\eqref{welldef} and concludes the proof.
\end{proof}

The proof of the following Proposition is analogous to the one of \cite[Proposition 2.7]{FS}, and uses Lemma \ref{IMT}.

\begin{prop}
Let $\{a_1,a_2,a_3,a_4\} \subset \R^3$ be defined as in \eqref{aaaa}, and denote $A_1 \doteq (a_1,a_2,a_3,a_4)$. Then, there exists $\delta > 0$ and for all $1\le \ell \le 4$ smooth maps
\[
\pi_\ell : (-\delta,\delta)^3\times B_\delta(A_1) \to \R^3
\]
with the following properties
\begin{itemize}
\item the map $t \mapsto \pi_\ell(t,X)$ is an embedding for each $X = (x_1,x_2,x_3,x_4) \in B_\delta(A_1)$;
\item $\pi_\ell(t,X) \in \{x_1,x_2,x_3, x_4\}^{\Lambda_\A}$ for all $t \in [0,\delta)^3$, $X = (x_1,x_2,x_3,x_4) \in B_\delta(A_1)$;
\item $\pi_\ell(0,X) = x_\ell$, for all $X = (x_1,x_2,x_3,x_4) \in B_\delta(A_1)$.
\end{itemize}
\end{prop}

With the help of the previous proposition, one can show the following result, see \cite[Theorem 2.8]{FS}.

\begin{theorem}\label{existenceinapp}
Let $\{a_1,a_2,a_3,a_4\} \subset \R^3$ be defined as in \eqref{aaaa}. Then, there exists a $\A$-in-approximation of $\{a_1,a_2,a_3,a_4\}$ for the operator $\A$ defined in \eqref{Axi}.
\end{theorem}
\subsection{Exact values}\label{exval}
In this section we give all the exact values needed to see that the set $\{ a_1, a_2, a_3, a_4 \}$ of \eqref{aaaa} \emph{forms a large $\R^3$-$T_4$ configuration} in the sense of Definition \ref{T4} and that \eqref{Pc} is uniquely solvable. 
\\
\\
The permutations $\sigma_i, 1\le i \le 3$ are:\
\begin{equation}\label{permut}
\begin{split}
&(\sigma_1(1),\sigma_1(2),\sigma_1(3),\sigma_1(4)) = (1,2,3,4),\\
&(\sigma_2(1),\sigma_2(2),\sigma_2(3),\sigma_2(4)) = (4,1,2,3), \\
&(\sigma_3(1),\sigma_3(2),\sigma_3(3),\sigma_3(4)) = (3,4,1,2).
\end{split}
\end{equation}
\noindent The values $p^{\sigma_i}$ for $1 \leq i \leq 3$ are:
\begin{align*}
p^{\sigma_1}= \frac{1}{15}
\left(\begin{array}{c} 2 \\ 4 \\ 8\end{array}\right)
\ \ \ 
p^{\sigma_2}= \frac{1}{65}
\left(\begin{array}{c} 18 \\ 27 \\ 8\end{array}\right)
\ \ \ 
p^{\sigma_3}= \frac{1}{175}
\left(\begin{array}{c} 64 \\ 27 \\ 36\end{array}\right).
\end{align*}
\noindent The values $c_\ell^{\sigma_i}$ for $1 \leq \ell \leq 4$ and $1 \leq i \leq 3$ are:
\begin{align*}
c^{\sigma_1}_1= \frac{1}{15}
\left(\begin{array}{c} -1 \\ -2 \\ -4 \end{array}\right)
\ \ \ 
c^{\sigma_1}_2= \frac{1}{15}
\left(\begin{array}{c} 7 \\ -1 \\ -2\end{array}\right)
\ \ \ 
c^{\sigma_1}_3= \frac{1}{15}
\left(\begin{array}{c} -4 \\ 7 \\ -1\end{array}\right)
\ \ \ 
c^{\sigma_1}_4= \frac{1}{15}
\left(\begin{array}{c} -2 \\ -4 \\ 7\end{array}\right),
\end{align*}
\begin{align*}
c^{\sigma_2}_1= \frac{1}{65}
\left(\begin{array}{c} -6 \\ -9 \\ 19 \end{array}\right)
\ \ \ 
c^{\sigma_2}_2= \frac{1}{65}
\left(\begin{array}{c} -4 \\ -6 \\ -9\end{array}\right)
\ \ \ 
c^{\sigma_2}_3= \frac{1}{65}
\left(\begin{array}{c} 19 \\ -4 \\ -6\end{array}\right)
\ \ \ 
c^{\sigma_2}_4= \frac{1}{65}
\left(\begin{array}{c} -9 \\ 19 \\ -4\end{array}\right),
\end{align*}
\begin{align*}
c^{\sigma_3}_1= \frac{1}{175}
\left(\begin{array}{c} -16 \\ 37 \\ -9 \end{array}\right)
\ \ \ 
c^{\sigma_3}_2= \frac{1}{175}
\left(\begin{array}{c} -12 \\ -16 \\ 37\end{array}\right)
\ \ \ 
c^{\sigma_3}_3= \frac{1}{175}
\left(\begin{array}{c} -9 \\ -12 \\ -16\end{array}\right)
\ \ \ 
c^{\sigma_3}_4= \frac{1}{175}
\left(\begin{array}{c} 37 \\ -9 \\ -12\end{array}\right).
\end{align*}
\\
\\
The values $k_\ell^{\sigma_i}$ are $k_\ell^{\sigma_i}= i+1$  for any $1 \leq \ell \leq 4$ and $1 \leq i \leq 3$.
\\
\\
The values $(\xi^{\sigma_i}_{2\ell-1}, \xi^{\sigma_i}_{2\ell}) \in \R^2$, for $1 \leq \ell \leq 4$ and $1 \leq i \leq 3$ are given by:
\begin{align*}
\left(\begin{array}{c} \xi_{1}^{\sigma_1} \\ \xi_2^{\sigma_1} \end{array}\right) = \left(\begin{array}{c} -14 \\ 5 \end{array}\right)  
\ \ \ 
\left(\begin{array}{c} \xi_{3}^{\sigma_1} \\ \xi_4^{\sigma_1} \end{array}\right) = \left(\begin{array}{c} 19 \\ -8 \end{array}\right) 
\ \ \ 
\left(\begin{array}{c} \xi_{5}^{\sigma_1} \\ \xi_6^{\sigma_1} \end{array}\right) = \left(\begin{array}{c} 11 \\ -14 \end{array}\right) 
\ \ \ 
\left(\begin{array}{c} \xi_{7}^{\sigma_1} \\ \xi_8^{\sigma_1} \end{array}\right) = \left(\begin{array}{c} -4 \\ -17 \end{array}\right) , 
\end{align*}
\begin{align*}
\left(\begin{array}{c} \xi_{1}^{\sigma_2} \\ \xi_2^{\sigma_2} \end{array}\right) = \left(\begin{array}{c} -7 \\ -3 \end{array}\right)  
\ \ \ 
\left(\begin{array}{c} \xi_{3}^{\sigma_2} \\ \xi_4^{\sigma_2} \end{array}\right) = \left(\begin{array}{c} 6 \\ 16 \end{array}\right) 
\ \ \ 
\left(\begin{array}{c} \xi_{5}^{\sigma_2} \\ \xi_6^{\sigma_2} \end{array}\right) = \left(\begin{array}{c} 2 \\ -17 \end{array}\right) 
\ \ \ 
\left(\begin{array}{c} \xi_{7}^{\sigma_2} \\ \xi_8^{\sigma_2} \end{array}\right) = \left(\begin{array}{c} -18 \\ 2 \end{array}\right) , 
\end{align*}
\begin{align*}
\left(\begin{array}{c} \xi_{1}^{\sigma_3} \\ \xi_2^{\sigma_3} \end{array}\right) = \left(\begin{array}{c} -7 \\ -14 \end{array}\right)  
\ \ \ 
\left(\begin{array}{c} \xi_{3}^{\sigma_3} \\ \xi_4^{\sigma_3} \end{array}\right) = \left(\begin{array}{c} -9 \\ 19 \end{array}\right) 
\ \ \ 
\left(\begin{array}{c} \xi_{5}^{\sigma_3} \\ \xi_6^{\sigma_3} \end{array}\right) = \left(\begin{array}{c} 6 \\ 18 \end{array}\right) 
\ \ \ 
\left(\begin{array}{c} \xi_{7}^{\sigma_3} \\ \xi_8^{\sigma_3} \end{array}\right) = \left(\begin{array}{c} -20 \\ -9 \end{array}\right) .
\end{align*}

\appendix

\section{The three state problem for operators of order 1}

Here we show how to infer the rigidity of the three state problem \eqref{problem} from the rigidity of the three state problem of the divergence proved in \cite{PP}. We are indebted to Guido De Philippis for making us realize that in many cases the study of operators of order one reduces to the study of the divergence operator, thus greatly simplifying our original proof of the following result.

\begin{prop}\label{rig}
Let $\A \in \op(1,m,n,N)$ and let $u$ be a solution to \eqref{problem} on the open connected set $\Omega$ for $s = 3$. Then, $u$ is constant.
\end{prop}
\begin{proof}
Consider $v \doteq u - a_1$. Then, $v$ solves \eqref{problem} with $\{a_1,a_2,a_3\}$ replaced by $\{0,b_1,b_2\}$, with $b_1 = a_2 - a_1,b_2 = a_3 - a_1,b_1-b_2 = a_2 - a_3 \notin \Lambda_{\A}$. Now consider the operator $\A' \in  \op(1,m,2,N)$ defined as
\[
\A'(z_1,z_2) \doteq \A(z_1 b_1 + z_2 b_2)
\]
for all $z_i \in L^\infty(\Omega)$, $i =1,2$. Defining $w \doteq (\chi_{E_1},\chi_{E_2})$, $E_i \doteq \{x \in \Omega: u(x) = b_i\}$ and $e_1 = (1,0), e_2 =(0,1)$, it is easy to see that $w$ solves
\begin{equation}\label{probspec}
\begin{cases}
w(x) \in \{0,e_1,e_2\}, &\text{ a.e. on }\Omega\\
\mathcal{A}'(w) = 0, &\text{ in the sense of distributions},\\
e_1,e_2, e_1-e_2 \notin \Lambda_{\A'},
\end{cases}
\end{equation}
and that $u$ is constant if and only if $w$ is constant. Since $\A' \in  \op(1,m,2,N)$, by Definition \ref{d_operator} it admits a representation of the form
\[
\mathcal{A'}(z_1,z_2) = MDz_1 + NDz_2 = \dv(Mz_1 + Nz_2),  
\]
for $M,N \in \R^{N\times m}$ and all bounded $(z_1,z_2)$. The latter and the fact that $e_1,e_2, e_1-e_2 \notin \Lambda_{\A'}$ easily imply that \eqref{probspec} is equivalent to the fact that $Z(x) = Mw_1 + Nw_2$ solves:
\[
\begin{cases}
Z(x) \in \{0,M,N\}, &\text{ a.e. on }\Omega\\
\dv(Z) = 0, &\text{ in the sense of distributions},\\
M,N, M-N \notin \Lambda_{\dv}.
\end{cases}
\]
By \cite{PP}, we know that $Z$ is constant, and hence also $w$ and $u$ must be constant.
\end{proof}

\bibliographystyle{plain}
\bibliography{Four}
\end{document}